\documentclass[a4paper]{article}
\usepackage[utf8]{inputenc}\UseRawInputEncoding{}
\usepackage[T1]{fontenc}
\usepackage{lmodern}
\usepackage{amsmath,amsthm,amssymb,amsfonts,geometry,mathrsfs}
\usepackage{hyperref,cite}\hypersetup{colorlinks,allcolors=blue,breaklinks=true}
\usepackage[numbers,square,sort&compress]{natbib}
\usepackage[pagewise]{lineno}
\DeclareMathAlphabet{\mathbbold}{U}{bbold}{m}n
{\theoremstyle{plain}\newtheorem{theorem}{Theorem}[section]\newtheorem{corollary}[theorem]{Corollary}\newtheorem{lemma}[theorem]{Lemma}\newtheorem{proposition}[theorem]{Proposition}}\numberwithin{equation}{section}{\theoremstyle{remark}\newtheorem{remark}{\bf Remark}}
\usepackage{xcolor}

\usepackage{bm,autobreak}
\newcommand{\colorref}[2][olive]{\hyperref[#2]{\color{#1}\ref*{#2}}}%
\newcommand{\bref}[1]{\texorpdfstring{\colorref[black]{#1}}{\ref*{#1}}}
\newcommand{\bbm}{\begin{pmatrix}}
\newcommand{\ebm}{\end{pmatrix}}
\DeclareMathAlphabet{\mathbbold}{U}{bbold}{m}{n}
\DeclareMathOperator{\divop}{div}

\title{Solutions with large number of peaks for a slightly supercritical nonlinear equation in dimension three}
\author{Yixing Pu\thanks{School of Mathematics and Statistics, Key Laboratory of Nonlinear Analysis \& Applications (Ministry of Education), Hubei Key Laboratory of Mathematical Sciences, Central China Normal University, Wuhan 430079, China.} \thanks{School of Mathematical Sciences, Key Laboratory of MEA (Ministry of Education), and Shanghai
Key Laboratory of PMMP, East China Normal University, Shanghai 200241, China. \newline \texttt{Email}: yxpu@stu.ecnu.edu.cn}}
\date{}

\begin{document}

\allowdisplaybreaks{}

\maketitle

\begin{abstract}
      We investigate the existence of solutions to the semilinear equation with a slightly supercritical exponent in dimension three,
      \begin{align*}
            -\Delta u=K(x) u^{5+\mu},\quad u>0 ~\text{in}~ \mathbf{B}, \quad u=0 ~\text{on}~ \partial \mathbf{B},
      \end{align*}
      where $\mu >0$, $\mathbf{B}$ is the unit ball in $\mathbb{R}^3$, $K(x)$ is a nonnegative radial function under suitable conditions on $K$. We prove the existence of positive multi-peak solutions for $\mu>0$ small enough. All peaks of our solutions approach the boundary $\partial\mathbf{B}$ as $\mu\rightarrow 0$. Moreover, the number of peaks varies with the parameter $\mu$ as $\mu$ goes to $0^+$. 
      Note that the case $n\geq 4$ was considered by Liu and Peng \cite{LiuPeng2016}.
\end{abstract}



\section{Introduction}

We investigate the existence of the following nonlinear equation:
\begin{align}\label{eq.basic_equation_K_p}
      -\Delta u=K(x) u^{p-1},\quad u>0 ~\text{in}~ \mathbf{B}, \quad u=0 ~\text{on}~ \partial \mathbf{B},
\end{align}
where $p>2$, $\mathbf{B}$ is the unit ball in $\mathbb{R}^n$, $n\geq3$; $K(x)=K(|x|)$ is a radial nonnegative continuous function on $\overline{\mathbf{B}}$.

For $K \equiv 1$, it is well known that \eqref{eq.basic_equation_K_p} possesses solutions  when $2<p<2^*$, where $2^* = \frac{2n}{n-2}$ is the critical Sobolev exponent. If $K$ is radial and nonincreasing in $\mathbf{r}$, the Poho\v{z}aev identity \cite{Pohozaev_1965} implies that there is no solution to problem \eqref{eq.basic_equation_K_p} for $p\geq 2^*$. For the case $2<p<2^*$, all solutions should be radially symmetric based on the moving plane method introduced by Gidas, Ni and Nirenberg \cite{Gidas_Ni_Nirenberg_1979}.

When $K$ is not nonincreasing, the solvability of \eqref{eq.basic_equation_K_p} is quite different. 
A typical case is $K(x)=|x|^{\alpha}$ with $\alpha>0$:

\begin{equation}\label{eq.basic_henon}
      -\Delta u=|x|^\alpha u^{p-1},\quad u>0 ~\text{in}~ \mathbf{B}, \quad u=0 ~\text{on}~\partial{} \mathbf{B},
\end{equation}
which is originally proposed by H\'{e}non \cite{Henon_1973} for rotating stellar structures.
Ni \cite{Ni1982} showed that $2^*+\frac{2\alpha}{n-2}$ is the new critical exponent for the existence of positive radial solutions of \eqref{eq.basic_henon}.
More generally, Ni proved the existence of positive radial solutions of \eqref{eq.basic_equation_K_p} with:
\begin{align*}
      \begin{aligned}
      & K ~\text{is radial nonnegative and H\"{o}lder continuous function},\\ 
      &K(x)=o(|x|^{\alpha}) ~\text{at}~ |x|=0 ~\text{for some}~ \alpha>0,
      \end{aligned}
\end{align*}
and $2<p<2^*+\frac{2\alpha}{n-2}$. 

Since the moving plane method does not work when $K$ is not decreasing in $\mathbf{r}$, it means that the behaviour of solutions may be different. Particularly, the existence of non-radial solutions can exist in this case. Smets, Su and Willem \cite{Smets_Willem_Su2002} obtained the existence of non-radial solutions of \eqref{eq.basic_henon} for $2<p<2^*$ and $\alpha$ large enough. For the slightly subcritical case $p=2^*-\mu$ and $\mu>0$ small,
 Cao and Peng \cite{CaoPeng2003} analyzed the asymptotic behaviour of ground state solution of \eqref{eq.basic_henon} for $\alpha>0$. The ground state blows up at a single point of the boundary $\partial\mathbf{B}$ up to subsequence as $\mu\rightarrow 0$, which deduces that the ground state solution is non-radial for $\mu>0$ small enough.

 Later on, Peng \cite{Peng2006}, Pistoia and Serra \cite{Pistoia_Serra_2007} proved the existence of multiple boundary peak solutions for problem \eqref{eq.basic_henon} in the slightly subcritical case.

For the critical exponent case, i.e., when $p=2^*$ for \eqref{eq.basic_equation_K_p}, 
Wei and Yan \cite{WeiYan2013} considered the existence of non-radial solutions to problem \eqref{eq.basic_equation_K_p}. More precisely, they showed that there are infinitely many non-radial solutions for the problem as $p=2^*$ with assumptions $n \geq 4$, $K(1) > 0$ and $K'(1) > 0$. Later, Hao, Chen and Zhang \cite{Hao2015} proved a similar result for problem \eqref{eq.basic_henon} with critical exponent $p=2^*$ in $\mathbb{R}^3$.

For more studies on H\'{e}non type problems, see \cite{ LiPeng2009,Gladiali_Grossi_Neves_2013,Gladiali_Grossi_2012, ChenZhouNi_2000,Byeon_Wang2006,Byeon_Wang2005,CaoPengYan2009,Hirano_2009,GuoLiLi_2017,GuoLiu_2022} and the references therein.

In this paper, we will focus on the slightly supercritical case $p=2^*+\mu$ and $\mu>0$ small for \eqref{eq.basic_equation_K_p}. In this case, Liu and Peng \cite{LiuPeng2016} proved that for $n\geq 4$, if $K\in C^{1}$ satisfies $K(1) > 0$ and $K'(1) > 0$, then there exists $\mu_0>0$ such that for $\mu\in(0,\mu_0)$, problem \eqref{eq.basic_equation_K_p} has a solution $u_{\mu}$ whose number of local maximal points is of the order $\mu^{-\frac{1}{n-1}}$ as $\mu \rightarrow 0$. Due to technical reason, it is difficult to extent Liu-Peng's approach to dimension three. A natural question is whether the result holds in $\mathbb{R}^3$. Namely, we consider the existence of non-radial solutions to \eqref{eq.basic_equation_K_p} in the slightly supercritical case when $n=3$:
\begin{align}\label{eq.basic_equation_n_3}
      -\Delta u=K(x) u^{5+\mu},\quad u>0 ~\text{in}~ \mathbf{B}\subset \mathbb{R}^3, \quad u=0 ~\text{on}~ \partial \mathbf{B},
\end{align}
where $K(x)=K(|x|)$ satisfies:
\begin{equation}\label{eq.K}\makeatletter\tagsleft@true\let\veqno\@@leqno\makeatother
      \begin{aligned}
            &K\in C^0\left(\left[0,1\right]\right) \cap C^{1,1}\left(\left(0,1\right]\right), \quad K(1)>0,\quad K'(1)>0.
      \end{aligned}\tag{H}
\end{equation}
Without loss of generality, we can assume that $K(1)=1$.

Our main result is as follows.
\begin{theorem}\label{th.1}
      Let $K$ satisfies \eqref{eq.K}. Then there exists $\mu_0>0$ such that for any $\mu \in (0,\mu_0)$, problem \eqref{eq.basic_equation_n_3}
      has a solution $u_\mu$ whose number of local strict maximal points is of the order $\mu^{-\frac{1}{2}}$
      as $\mu \rightarrow 0$.
\end{theorem}

As a special case, consider $K\left(x\right)=\left|x\right|^{\alpha}$ with $\alpha>0$, then we have 
\begin{corollary} Given $\alpha>0$, there exists $\mu_0>0$ such that for any $\mu \in (0,\mu_0)$, the H\'{e}non equation
      \begin{equation}\label{eq.basic_henon_n_3}
            -\Delta u=|x|^\alpha u^{5+\mu},\quad u>0 ~\text{in}~ \mathbf{B}\subset \mathbb{R}^3, \quad u=0 ~\text{on}~\partial \mathbf{B},
      \end{equation}
      where $\mathbf{B}$ is a unit ball in $\mathbb{R}^3$, has a solution $u_\mu$ whose number of local maximal points is of the order $\mu^{-\frac{1}{2}}$ as $\mu \rightarrow 0$.
\end{corollary}

We will use the finite-dimensional reduction method and prove the existence of multi-peak solutions which blow up on the boundary. Comparing to Liu-Peng's construction \cite{LiuPeng2016}, we will choose different ansatz by adapting some ideas of \cite{Hao2015}.

This paper has the following structure. In Section \ref{sec.2}, we construct the ansatz. 
In Section \ref{sec.new3}, we analyze the linearized problem near the ansatz. In Section \ref{sec.3}, we consider the associated nonlinear problem and prove our main result. The energy estimates and some basic estimates are postponed in the appendices.

\section{The Multi-peak Ansatz}\label{sec.2}
We will construct the $k$-peak ansatz for problem \eqref{eq.basic_equation_n_3} with 
\begin{align}\label{eq.def_k_mu}
      k = \big\lfloor\mu^{-\frac{1}{2}}\big\rfloor, \quad ~\text{for}~ \mu>0 ~\text{small enough},
\end{align}
 where $\left\lfloor x \right\rfloor$ represents the greatest integer less than or equal to $x\in \mathbb{R}$.


\subsection{The $k$-peak ansatz}\label{subsec.Ansatz_def}
\ \par
Denote $\Phi(x):=\frac{3^\frac{1}{4}}{{\left(1+|x|^2\right)}^\frac{1}{2}}$ which is the unique positive solution of
\begin{align}\label{eq.peak_equation}
      -\Delta\Phi=\Phi^5 \quad \text{in}~ \mathbb{R}^3,
\end{align}
up to the translation and scaling.
Let 
\begin{align}\label{eq.m_def}
      \mathbf{m}_i := \Bigl(\cos\frac{2\pi i}{k},\sin\frac{2\pi i}{k},0\Bigr),\quad \text{for}~ i = 0,1,\cdots,k-1, 
\end{align}
which means the vertices of a regular $k$-sided polygon on $\mathbb{R}^2 \times \{0\}$.

Denote $(\lambda^*,\sigma^*)$ as unique root of equation \eqref{eq.L_L_prime_algebraic_system} below. For any 
\begin{align}\label{eq.independent_parameters}
      \lambda\in\Bigl[\frac{1}{2}\lambda^*,2\lambda^*\Bigr], \quad \sigma\in\Bigl[\frac{1}{2}\sigma^*,2\sigma^*\Bigr],
\end{align}
we define
\begin{align}\label{eq.dependent_parameters}
      \varepsilon=\lambda^{-1}\mu, \quad r= 1-\sigma\mu^\frac{1}{2},
\end{align}
and for $i=0,1,\cdots,k-1$,
\begin{align}\label{eq.U_i_U_i_star}
            U_i(x) := \frac{K^{-\frac{1}{4}}(r)}{\varepsilon^\frac{1}{2}}\Phi \Bigl(\frac{x-r\mathbf{m}_i}{\varepsilon}\Bigr), \quad U_{i^*}(x) := \frac{K^{-\frac{1}{4}}(r)}{\varepsilon^\frac{1}{2}}\Phi\Bigl(\frac{rx-\mathbf{m}_i}{\varepsilon}\Bigr),
\end{align}
which satisfy
\begin{align*}
      -\Delta U_i=K(r)U_i^5 ~\text{in}~ \mathbb{R}^3, \quad -\Delta U_{i^*}=r^2K(r) U_{i^*}^5 ~\text{in}~ \mathbb{R}^3.
\end{align*}

Our $k$-peak ansatz to \eqref{eq.basic_equation_n_3} is given by
\begin{equation}\label{eq.W_def}
      W:=\sum_{i=0}^{k-1}(U_i-U_{i^*}).
\end{equation}
Note that the above choice ensures that 
\begin{align*}
      U_i-U_{i^*}=0 ~\text{on}~ \partial \mathbf{B}, ~\text{hence}~ W=0 ~\text{on}~ \partial \mathbf{B}.
\end{align*}


For simplicity, we omit in general the dependence on $\lambda,\sigma,\mu$.

\subsection{Idea of the construction of a $k$-peak solution}
\ \par
We rewrite $\mathbb{R}^3$ as $\mathbb{C}\times \mathbb{R}$ and consider the following functional space. 
Let $k\in \mathbb{N}, k \geq 2$ and
\begin{align*}
      \begin{aligned}
           \mathcal{H}_k := \Big\{& f \in H_0^1\left(\mathbf{B}\right): ~ f=0 ~\text{on}~ \partial \mathbf{B}, ~\text{and for any}~ z\in \mathbb{C}, ~ y \in \mathbb{R},
            \\
            & f(z,y) = f(-z,y) = f(z e ^{\mathbf{i} \frac{2\pi}{k}},y)=f(z,-y) \Big\},
      \end{aligned}
\end{align*}
where $\mathbf{i}$ denotes the imaginary unit.

We will seek a solution to \eqref{eq.basic_equation_n_3} by $v= W + \varphi$ .
More precisely, we will consider the nonlinear problem:
\begin{align}\label{eq.nonlinear_varphi_equivalent}
     \quad \mathcal{L}\varphi= F+N(\varphi),
\end{align}
where $\varphi\in \mathcal{H}_k$, the linear operator $\mathcal{L}$ is
\begin{align}
      \mathcal{L}\varphi := -\Delta\varphi-(5+\mu)K(x) W^{4+\mu} \varphi,
\end{align}
and
\begin{align}\label{eq.linear_error}
      \begin{aligned}
            F &:= \Delta W+K\left(x\right) W^{5+\mu},\\
            N(\varphi) &:= K\left(x\right)\left[\left|W+\varphi\right|^{5+\mu}-W^{5+\mu}-\left(5+\mu\right)W^{4+\mu}\varphi\right].
      \end{aligned}
\end{align}
It is hard to solve \eqref{eq.nonlinear_varphi_equivalent} directly because the kernel of $\mathcal{L}$ is nontrivial. In fact, we will first consider the constraint problem:
\begin{equation}\label{eq.c_1_c_2_equation}
      (c_1,c_2,\varphi)\in \mathbb{R}\times \mathbb{R}\times \mathcal{H}_{k,0}, \quad \mathcal{L}\varphi= F+N(\varphi)+c_1\Delta W_\lambda +c_2\Delta W_\sigma ,
\end{equation}
where $W_\lambda = \partial_\lambda W$, $W_\sigma = \partial_\sigma W$ and $\mathcal{H}_{k,0}$ is the closed linear subspace in $\mathcal{H}_k$ defined by:
\begin{align*}
            \mathcal{H}_{k,0}  := \left\{f\in \mathcal{H}_k: ~ \left< f,\Delta W_{\lambda} \right> =\langle f,\Delta W_{\sigma}\rangle =0 \right\}.
\end{align*}
Here and in the following, we use the notations $f_\lambda = \partial_\lambda f$, $f_\sigma = \partial_\sigma f$ 
and $ \left< f,g \right>$ means the inner product in $L^2(\mathbf{B})$. i.e.
\begin{align*}
      \left< f,g \right>=\int_{\mathbf{B}}f g dx.
\end{align*} 
In fact, we will choose suitable $\lambda(\mu)$ and $\sigma(\mu)$ such that $c_1=c_2=0$ and for the solution $\varphi$ to \eqref{eq.c_1_c_2_equation}, hence $\varphi$ is a solution to \eqref{eq.nonlinear_varphi_equivalent}. 


\subsection{More understanding of the ansatz}\label{subsec.notations}

Based on the rotational symmetry, we will decompose $\mathbf{B}$ as follows.
\begin{align}
      \Omega_i:=\Bigl\{x\in \mathbf{B}:x\cdot \mathbf{m}_i=\max_{0\leq j\leq k-1} (x \cdot \mathbf{m}_j)\Bigr\}, ~\text{for}~ i=0,1,\cdots,k-1.
\end{align}
In this paper $x \cdot y$ means the Euclidean inner product in $\mathbb{R}$. Due to the rotational periodicity in $\mathcal{H}_k$, we will often focus on $\Omega_0$.
To understand more precisely of $W$, we introduce the following notation:
\begin{equation}
      R:= W-U_0 = \sum_{i=1}^{k-1}\left(U_i-U_{i^*}\right)-U_{0^*}.
\end{equation}
We rewrite $R$ as
\begin{equation}
      R(x) = \Phi(0)K^{-\frac{1}{4}}(r)E(x),
\end{equation}
where
\begin{align}
      \quad E(x):= \frac{\lambda}{\mu^\frac{1}{2}}\left(\sum_{i=1}^{k-1}\frac{1}{d_i (x)}-\sum_{i=0}^{k-1}\frac{1}{d_{i^*}(x)}\right),
\end{align}
and
\begin{align*}
      d_i(x):={ \sqrt{1+ \frac{|x-r\mathbf{m}_i|^2}{\varepsilon^2}}}, \quad  d_{i^*}(x) :={ \sqrt{1+\frac{|rx-\mathbf{m}_i|^2}{\varepsilon^2}}}.
\end{align*}

For $\rho>0$, denote $\Vert \cdot \Vert_\rho $ as the weighted norm: 
\begin{align*}
      \Vert f\Vert_\rho := \sup_{x \in \mathbf{B}}\frac{|f(x)|}{\omega_\rho(x)},
\end{align*}
where the weight $\omega_\rho(x)$ is given by
\begin{align}
      \omega_\rho(x):= {\sum_{i=0}^{k-1}\frac{1}{d_i^\rho(x)}}.
\end{align}
For convenience, we also use $\omega_\rho^*(x)$ as follows.
\begin{align}
      \omega_\rho^*(x):= {\sum_{i=0}^{k-1}\frac{1}{d_{i^*}^\rho(x)}}.
\end{align}
.

Similar to \cite[Lemma 2.3]{Hao2015}, we get the estimates of $E(r\mathbf{m}_0)$ and $\nabla E(r\mathbf{m}_0)$ by direct calculation.
\begin{lemma}
      As $\mu$ goes to $0^+$, there hold
      \begin{align}\label{eq.E_x_0_and_nabla_E_x_0}
            \begin{aligned}
            E(r\mathbf{m}_0)&=\frac{1}{2}L(\sigma)+O\bigl(\mu^\frac{1}{2}\bigr),\\
            \mathbf{m}_0\cdot \nabla E(r\mathbf{m}_0) &= \mu^{-\frac{1}{2}}\Bigl( -\frac{1}{4}L'(\sigma)+O\bigl(\mu^\frac{1}{2}\bigr) \Bigr),
            \end{aligned}
      \end{align}
      where 
      \begin{align}\label{eq.L_L_prime_def}
            L(\sigma):= \sum_{j\in \mathbb{Z}\backslash \{0\}}\Biggl(\frac{1}{\left|j\pi\right|}-\frac{1}{{\bigl[{\left(j\pi\right)}^2+\sigma^2\bigr]}^{\frac{1}{2}}}\Biggr),
            \quad \text{hence}~ L'(\sigma)=\sum_{j\in \mathbb{Z}}\frac{\sigma}{{\bigl[\sigma^2+{\left(j\pi\right)}^2\bigr]}^\frac{3}{2}}.
      \end{align}
\end{lemma}

In this paper, $O(f)$ means a generic quantity satisfying $|O(f)|\leqslant C |f|$, for some constant $C$ which is independent of $\lambda,\sigma,\mu$. 

Finally, we denote
\begin{align}\label{eq.A_A_1_A_2}
           A_1:=\frac{1}{6}\int_{\mathbb{R}^3}\Phi^6(x)dx,\quad A_2:=\frac{1}{4}\int_{\mathbb{R}^3}\Phi(0)\Phi^5(x)dx,
      \end{align}
      where $\Phi=\frac{3^\frac{1}{4}}{{\left(1+|x|^2\right)}^\frac{1}{2}}$ is the unique positive solution of \eqref{eq.peak_equation}. Consider the system of $(\lambda,\sigma)$:
      \begin{equation}\label{eq.L_L_prime_algebraic_system}
            L(\sigma)=0, \quad K'(1)A_1 \lambda=A_2 L'(\sigma),
      \end{equation}
      where $L$ is defined in \eqref{eq.L_L_prime_def}.
According to Section 2.1 in \cite{Hao2015},  \eqref{eq.L_L_prime_algebraic_system} possesses a unique positive root which is denoted by $(\lambda^*,\sigma^*)$.

\subsection{Main estimates of $R$}


\begin{lemma}
      As $\mu$ goes to $0$, $R$ satisfies
      \begin{align}\label{eq.R_estimates_parameters}
            |R(x)|+ |R_{\lambda}(x)|+ |R_{\sigma}(x)| \leq C \left[\frac{1}{\mu^{\frac{1}{2}}d_{0^*}(x)}+
            \sum_{i=1}^{k-1}\frac{1}{\mu d_{i^*}^2(x)}\right], \quad \text{for}~ x\in \overline{\Omega}_0,
      \end{align}
      and
      \begin{align}\label{eq.R_estimates_x}
            \mu^\frac{1}{2}\Vert \nabla R\Vert_{L^{\infty}(\overline{\Omega}_0)}+\mu\Vert \nabla^2R\Vert_{L^{\infty}(\overline{\Omega}_0)} \leq C.
      \end{align}
\end{lemma}
\begin{proof}
      Note that 
      \begin{align*}
            U_i (x)=\frac{\lambda^{\frac{1}{2}}K^{-\frac{1}{4}}(r)\Phi(0)}{\mu^\frac{1}{2} d_i(x)},\quad U_{i^*} (x)=\frac{\lambda^{\frac{1}{2}}K^{-\frac{1}{4}}(r)\Phi(0)}{\mu^\frac{1}{2} d_{i^*}(x)}.
      \end{align*}
      It can be obtained that
       \begin{align}\label{eq.U_i_lambda_sigma_estimates}
            \begin{aligned}
                  {(U_i)}_{\lambda}&=\frac{U_i}{2\lambda}\Bigl( \frac{2}{d_i^2}-1\Bigr), \quad
                  {(U_{i^*})}_{\lambda}=\frac{U_{i^*}}{2\lambda}\Bigl(\frac{2}{d_{i^*}^2}-1\Bigr),\\
                  {(U_i)}_{\sigma}&=\frac{\lambda U_i}{\mu^{\frac{1}{2}}}
                  \Bigl(\frac{r \mathbf{m}_i-x}{\varepsilon}\cdot \frac{\mathbf{m}_i}{d_i^2}+\frac{\varepsilon}{4}\frac{K'(r)}{K(r)} \Bigr)=\frac{O(U_i)}{\mu^{\frac{1}{2}} d_i},
                  \\[3.5pt]
                   {(U_{i^*})}_{\sigma}&=\frac{\lambda U_{i^*}}{\mu^{\frac{1}{2}}}
                   \Bigl(\frac{rx -\mathbf{m}_i}{\varepsilon}\cdot \frac{x}{d_{i^*}^2}+\frac{\varepsilon}{4}\frac{K'(r)}{K(r)} \Bigr) = \frac{O(U_{i^*})}{\mu^{\frac{1}{2}} d_{i^*}}.
            \end{aligned}
      \end{align}

      Together with the estimates \eqref{eq.omega_rho_omega_rho_star_estimates} and \eqref{eq.d_i_comparing_d_i_star_estimates} in Appendix \ref{appendix}, the estimates for $R_\lambda$ and $R_\sigma$ are valid. Combining direct differentiation, the estimates for $R$, $\nabla R$ and $\nabla^2 R$ hold. The proof is done.
\end{proof}


\subsection{Error estimates for the ansatz}
Now we give the estimates for the error $F=\Delta W+K(x) W^{5+\mu}$. 
\begin{lemma}
     For $\mu>0$ small enough, fix $\rho\in [0,3)$, there is $C>0$ such that
      \begin{equation}\label{eq.F_estimates}
            |F(x)| \leq \frac{C}{d_0^{\rho+2}(x)} \left\{\mu^{-2}\left|L\left(\sigma\right)\right|+\mu^{\frac{\min\{1,2-\rho \}}{2}-2} \left|\ln\mu\right| \right\},~\text{for}~ x\in \overline{\Omega}_0.
      \end{equation}
      In particular, 
      \begin{align*}
      \left\Vert F \right\Vert_{\rho+2}\leq C \mu^{\frac{\min\{0,2-\rho\}}{2}-2}\left|\ln\mu\right|.
      \end{align*}

\end{lemma}
\begin{proof}
      Assume that $x\in\overline{\Omega}_0$ and $\rho \in [0,3)$. Using $W=U_0+R$ and $-\Delta U_0=K(r)U_0^5$, we obtain
      \begin{align}\label{eq.temp_F_x_F_1}
            \begin{aligned}
                  F = \left(K(x)U_0^{5+\mu}-K(r)U_0^5\right)+(5+\mu)K(x)U_0^{4+\mu}R+\Delta R +F_1,
            \end{aligned}
      \end{align}
      where 
      \begin{align*}
            F_1 := K(x)\left(W^{5+\mu}-U_0^{5+\mu}- (5+\mu)U_0^{4+\mu}R\right).
      \end{align*}
      We estimate each term in the right-hand side of \eqref{eq.temp_F_x_F_1}.

      First, since $\frac{1}{d_{i^*}}\leq \frac{1}{d_{0^*}^\theta d_{i^*}^{1-\theta}}$ for $\theta \in [0,1)$,
      by \eqref{eq.R_estimates_parameters} and \eqref{eq.omega_rho_omega_rho_star_estimates}, there holds
      \begin{equation}
            \label{eq.epsilon_R_theta_estimates}
            \left|R \right| \leq \frac{C\mu^{-\frac{\theta}{2}}}{d_{0^*}^\theta}\quad \text{for all}~\theta\in [0,1).
      \end{equation}
   Here we discuss two subcases:

      Case (i). For $\rho\in[1,3)$, or $d_0\geq \mu^{-\frac{1}{2}}$.
      Ccombined with the fact
      \begin{align}
            d_0^\mu\rightarrow 1 ~\text{as}~ \mu\rightarrow 0,
      \end{align}
      by selecting $\theta = 2\rho - 1$, we get
      \begin{align}\label{eq.F_1_subcases_temp_1}
            F_1 \leq C\left[ \left|R\right|^{5+\mu}+U_0^{3+\mu}R^2\right] \leq \frac{C}{d_0^3}
            \frac{\mu^{-\frac{3}{2}-\theta}}{d_{0^*}^{2\theta}}
            \leq \frac{C\mu^\frac{-2-\rho}{2}}{d_0^{\rho+2}}.
      \end{align}

      Case (ii). For $\rho\in \left[0,1\right)$ and $d_0\leq \mu^{-\frac{1}{2}}$. Using \eqref{eq.R_estimates_x}, we see that
      \[
             R^2(x)
            = R^2(r\mathbf{m}_0)+O(\mu^\frac{1}{2}d_0)
            = R^2(r\mathbf{m}_0)+O\bigl(\mu^{\frac{1-\rho}{2}}d_0^{1-\rho}\bigr).
      \]
      Hence
      \[
            U_0^{3+\mu} R^2 = U_0^{3+\mu} R^2\left(r\mathbf{m}_0\right) + \frac{O\bigl(\mu^\frac{-1-\rho}{2}\bigr)}{d_0^{\rho+2}}.
      \]
      Then
      \begin{align}\label{eq.F_1_subcases_temp_2}
      F_1=\frac{(5+\mu)(4+\mu)}{2}U_0^{3+\mu}R^2(r\mathbf{m}_0)+\frac{O\bigl(\mu^\frac{-1-\rho}{2}\bigr)}{d_0^{\rho+2}}
      =\frac{O\bigl(\mu^{\frac{\min \left\{1,2-\rho\right\}}{2}-2}\bigr)}{d_0^{\rho+2}}.
      \end{align}
     Combining \eqref{eq.F_1_subcases_temp_1} and \eqref{eq.F_1_subcases_temp_2}, we have 
      \begin{align*}
            F_1 = \frac{(5+\mu)(4+\mu)}{2}U_0^{3+\mu}R^2(r\mathbf{m}_0) \mathbbold{1}_{\{\sqrt{\mu} d_0<1\}}
            +\frac{O\bigl(\mu^\frac{-2-\rho}{2}\bigr)}{d_0^{\rho+2}}.
      \end{align*}
      Consequently,
      \begin{align}\label{eq.F_temp_estimates}
            \begin{aligned}
                  F  ={}&\left(K\left(x\right)U_0^{\mu}-K\left(r\right)\right)U_0^5+\left(5+\mu\right)K\left(x\right)U_0^{4+\mu}R\\
                  & +\Delta R +
                  \frac{(5+\mu)(4+\mu)}{2}U_0^{3+\mu}R^2(r\mathbf{m}_0) \mathbbold{1}_{\{\sqrt{\mu} d_0<1\}}
                  +\frac{O\bigl(\mu^\frac{-2-\rho}{2}\bigr)}{d_0^{\rho+2}}.
            \end{aligned}
      \end{align}

      For the second term of the right-hand side, since $K(r)\left(1-r^2\right)\sum_{0\leq i \leq k-1}U_{i^*}^5<2\sigma U_{0^*}^5$, we obtain from \eqref{eq.omega_rho_omega_rho_star_estimates} and \eqref{eq.d_i_comparing_d_i_star_estimates} that
      \begin{align*}
            \begin{aligned}
            |\Delta R|  = K(r)\biggl|r^2\sum_{i=0}^{k-1}U_{i^*}^5-\sum_{i=1}^{k-1}U_{i}^5\biggr| \leq &
            \frac{C}{\mu^\frac{5}{2}}
            \biggl(\frac{1}{d_{0^*}^5}+\sum_{i=1}^{k-1}\frac{1}{\mu^\frac{1}{2}d_{i^*}^6} \biggr) \\[3pt]
            \leq {}&\frac{C}{\mu^\frac{5}{2}d_{0^*}^{\rho+2}} 
            \left(\frac{1}{d_{0^*}^{3-\rho}}+\sum_{i=0}^{k-1}\frac{1}{\mu^\frac{1}{2}d_{i^*}^{4-\rho}}\right)
            \\[3pt]
            \leq{}&\frac{C}{\mu^\frac{5}{2}}\frac{\mu^\frac{3-\rho}{2}}{d_{0^*}^{\rho+2}}.    
            \end{aligned}
      \end{align*}
Here we use $\lambda \varepsilon = \mu$ by $\varepsilon = \lambda^{-1}\mu$ and $\lambda$ in \eqref{eq.independent_parameters}. 

 Since $K\left(x\right)U_0^{\mu}-K(r)=\left(K\left(x\right)-K\left(r\right)\right)+K\left(x\right)\left(U_0^\mu-1\right)=O\left(\mu d_0\left|\ln \mu\right|\right)$, we have
      \begin{align*}
            \begin{aligned}
                  \left(K\left(x\right)U_0^{5+\mu}-K\left(r\right)U_0^5\right)&=\frac{O\left(\mu\left|\ln \mu\right|\right)}{\varepsilon\frac{5}{2}d_0^4}\\[10pt]
                  &=\begin{cases}
                        \displaystyle\frac{O\left(\mu^{-\frac{3}{2}}\left|\ln \mu\right|\right)}{d_0^{\rho+2}}            & \text{if}~ \rho\in \left[0,2\right), \vspace{10pt} \\
                        \displaystyle\frac{O\left(\varepsilon^{3-\rho}{\left( \varepsilon d_0 \right)}^{\rho-2}\right)}{\varepsilon^\frac{5}{2}d_0^{\rho+2}}
                        =\frac{O\left(\mu^{\frac{1}{2}-\rho}\left|\ln \mu\right|\right)}{ d_0^{\rho+2}} & \text{if}~\rho\in\left[2,3\right).
                  \end{cases}
            \end{aligned}
      \end{align*}
      \par\noindent
     Finally, we estimate $U_0^{4+\mu}R$ by the following two subcases.

      Case (i). If $\rho\in\left[0,2\right)$ and $d_0\leq \mu^{-\frac{1}{2}}$, then
      \begin{align*}
            \begin{aligned}
                  U_0^{4+\mu}R\leq & C\biggl(\frac{\mu^\frac{1}{2}\left|E(r\mathbf{m}_0)\right|}{ \varepsilon^\frac{5}{2} d_0^4}
                  +\frac{\mu\left|\ln \mu\right|}{ \varepsilon^\frac{5}{2} d_0^3}\biggr)\\
                  \leq & C\frac{1}{d_0^{\rho+2}}\left(\left|E(r\mathbf{m}_0)\right|\mu^{-2}
                  +\mu^{\frac{\min \left\{1,2-\rho\right\}}{2}-2}\left|\ln \mu\right| \right).   
            \end{aligned}
      \end{align*}
      Recall that $E(r\mathbf{m}_0)=\frac{1}{2}L\left(\sigma\right)+O\left(\mu^\frac{1}{2}\right)$ in \eqref{eq.E_x_0_and_nabla_E_x_0}. Thus it is done.
      
      Case (ii). If $\rho\in[2,3)$, or $\rho\in[0,2)$ and $d_0\geq \mu^{-\frac{1}{2}}$, then
      \[
            U_0^{4+\mu}R=\frac{O\bigl(\mu^\frac{1-\theta}{2}\bigr)}{\varepsilon^\frac{5}{2}d_0^4\cdot d_{0^*}^\theta}=
            \frac{O\bigl(\mu^\frac{3-\rho}{2}\bigr)}{\varepsilon^\frac{5}{2} d_0^{\rho+2}}=
            \frac{O\bigl(\mu^\frac{-1-\rho}{2}\bigr)}{ d_0^{\rho+2}}.
      \]
Combining all these estimates above, we obtain the assertion of Lemma \ref{lemma.energy_estimates}.
\end{proof}

\subsection{Further estimates for \texorpdfstring{$W$, $W_\lambda$, $W_\sigma$}{\$W\$, \$W\_\textbackslash{} lambda\$, \$W\_\textbackslash{} sigma\$}}
 We present some other useful estimates for our ansatz.
\begin{lemma}\label{lemma.W_norm_estimates}
      The following estimates hold:\vspace{5pt}\par
      (i) for any $\rho\geq\frac{1}{2}$ and $f\in  \mathcal{H}_k$, $\| f W^{4+\mu} \|_{\rho+3} \leq C \| \mu^{-2} f \|_\rho$,\vspace{10pt}\par
      (ii) $\left\Vert \Delta W_\lambda\right\Vert_5+\mu^\frac{1}{2}\left\Vert \Delta W_\sigma\right\Vert_6=O\left(\mu^{-\frac{5}{2}}\right)$.
\end{lemma}
\begin{proof} Recall that $\Vert \cdot \Vert_\rho $ means the weighted norm: 
      \begin{align*}
            \Vert f\Vert_\rho = \sup_{x \in \mathbf{B}}\frac{|f(x)|}{\omega_\rho(x)}.
      \end{align*}
      First, we estimate $\| f W^{4+\mu} \|_{\rho+3}$.
      Using $W=U_0+R$ and \eqref{eq.epsilon_R_theta_estimates}, we have, for any $\theta\in[0,1)$ and $x\in\Omega_0$,
      \begin{equation}
            \label{eq.W_d_theta_mu}
           W=\frac{O(1)}{ \varepsilon^\frac{1}{2} d_0}+\frac{\mu^{-\frac{\theta}{2}}}{ d_{0^*}^\theta}.
      \end{equation}    
      Hence, let $x\in \Omega_0$ and fix $\theta=\frac{5}{8}$, then $\rho+3\geq 4\theta\geq \frac{1}{2}$. Combined with \eqref{eq.omega_rho_omega_rho_star_estimates}, it holds
      \begin{equation*}
            \frac{\mu^2 f W^{4+\mu}}{\Vert f\Vert_\rho}=\frac{O(\omega_{\rho})}{d_0^{4\theta}}=O(\omega_{4\theta}\omega_{\rho}) = O\bigl(\omega_{\rho+3}^{\frac{6}{2\rho+5}}\omega_{\rho+3}^{\frac{2\rho-1}{2\rho+5}} \bigr)=O(\omega_{\rho+3}).
      \end{equation*}
      The first assertion of the lemma thus follows.\par
      Next, note that $-\Delta W_\lambda ={(-\Delta W)}_\lambda=O(1)\varepsilon^{-\frac{5}{2}}\sum_{i=0}^{k-1}d_i^{-5}$, we hence find that $\Vert\Delta W_\lambda\Vert_5 = O(\mu^{-\frac{5}{2}})$. Similarly, we find that $\Vert\mu^\frac{1}{2}\Delta W_\sigma\Vert_6=O(\mu^{-\frac{5}{2}})$.
      This completes the proof.
\end{proof}
Now, we show the almost orthogonality of $W_\lambda$ and $W_\sigma$ in $H_0^1(\mathbf{B})$.
\begin{lemma}\label{lemma.M_estimates}
      There are positive  constants $a_1$ and $a_2$ such that
      \begin{equation}\label{eq.M_estimates}
            M := \mu^{\frac{1}{2}}
            \bbm \lambda^2\left< \nabla W_\lambda,\nabla W_\lambda \right>
            & \lambda^{-1}\mu^{\frac{1}{2}}\left< \nabla W_\sigma,\nabla W_\lambda \right>
            \\ \lambda^{-1}\mu^{\frac{1}{2}}\left< \nabla W_\sigma,\nabla W_\lambda \right>
            &  \lambda^{-2} \mu \left< \nabla W_\sigma,\nabla W_\sigma \right>
            \ebm
            = \bbm a_1 & 0 \\ 0 & a_2 \ebm +O\bigl(\mu^\frac{1}{2}\bigr).
      \end{equation}
\end{lemma}
\begin{proof}
      For $l,t \in \left\{\lambda,\sigma\right\}$, Using integration by parts and symmetry,
      \begin{equation}
            \label{eq2.31}
            \mu^\frac{1}{2}\left<\nabla W_l,\nabla W_t \right>=\mu^\frac{1}{2}\int_{\mathbf{B}}\left(-\Delta W_l\right)W_t dx=
            \int_{\mathbf{B}}{\left[K(r)(U_0^5-r^2U_{0^*}^5)\right]}_l W_t dx.
      \end{equation}
      Note that $d_0 \thicksim d_{0^*}$ in $\mathbf{B}\backslash\mathbf{B}_0$, where 
      \begin{align*}
            \mathbf{B}_0 :=  \left\{x:|x-r\mathbf{m}_0|<\sigma_0 \mu^\frac{1}{2}\right\}\subset\Omega_0,
      \end{align*}
      which is defined as in \eqref{eq.B_def}. Also $|U_{i\lambda}|+\mu^\frac{1}{2}|U_{i\sigma}|+|U_{i^*\lambda}|+|U_{i^*\sigma}|=O(U_i) ~\text{in}~ \mathbf{B}$. Hence,
      for $l=\lambda$ or $l=\sigma$, by \eqref{eq.d_i_d_i_star_estimates}, we get
      \begin{align*}
            \begin{aligned}
                  &\int_{\mathbf{B}\backslash\mathbf{B}_0}\Bigl|{\left[K(r)(U_0^5-r^2U_{0^*}^5)\right]}_l \Bigr|\left(\left|W_\lambda\right|+\mu^\frac{1}{2}W_\sigma\right)dx \\
                  &\hspace{2em} \leq C\sum_{i=0}^{k-1}\int_{\mathbf{B}\backslash\mathbf{B}_0}\frac{1}{d_{0^*}d_i}\frac{dx}{\varepsilon^3} \leq C \sum_{i=0}^{k-1}\frac{\mu}{d_{0^*}(r\mathbf{m}_i)} \leq C\sum_{i=0}^{k-1}\frac{\mu}{d_{i^*}(r\mathbf{m}_0)}=O(\mu).
            \end{aligned}
      \end{align*}
      Using \eqref{eq.R_estimates_parameters}, we obtain that for $t=\lambda$ or $t=\sigma$,
      \begin{equation*}
            \begin{aligned}
                  &\int_{\mathbf{B}_0}\left\{\left|{[K(r)(U_0^5-r^2 U_{0^*}^5)]}_\lambda\right|+\mu^\frac{1}{2}\left|{[K(r)(U_0^5-r^2U_{0^*}^5)]}_\sigma\right|\right\}|R_t|dx
                  \\
                  &\hspace{3em} =O\bigl(\mu^\frac{1}{2}\bigr)\int_{\mathbf{B}_0}\frac{1}{d_0^5}\frac{dx}{\varepsilon^3}=O\bigl(\mu^\frac{1}{2}\bigr).
            \end{aligned}
      \end{equation*}
      The assertion of the lemma thus follows from the identities, ignoring the $K (r)$ factor,
      \begin{equation*}
            \begin{aligned}
                  \lambda^2\int_{\mathbb{R}^3}U_0^4 {\left(U_{0}\right)}_{\lambda}^2 dx=a_1 := 
                  \int_{\mathbb{R}^3}\Phi^4(x)\left|\frac{1}{2}\Phi (x)+x\cdot \nabla \Phi (x)\right|^2dx, \\
                  \lambda^{-2}\mu \int_{\mathbb{R}^3}U_0^4 {\left(U_{0}\right)}_{\sigma}^2 dx=a_2 := 
                  \frac{1}{3}\int_{\mathbb{R}^3}\Phi^4(x)\left|\nabla \Phi (x)\right|^2dx.
            \end{aligned}
      \end{equation*}
      So we are done.
\end{proof}

\section {Estimates for the Linear Operator $\mathcal{L}$}\label{sec.new3}

\subsection{Smallness of \texorpdfstring{$\mathcal{L}W_\lambda$ and $\mathcal{L}W_\sigma$}{\$\textbackslash{} mathcal\{L\}W\_\textbackslash{} lambda\$ and \$\textbackslash{} mathcal\{L\}W\_\textbackslash{} sigma\$}}
Recall that $ \mathcal{L}\varphi = -\Delta\varphi-(5+\mu)K(x) W^{4+\mu} \varphi$.
\begin{lemma}\label{lemma.LW_lambda_sigma}
      $\mathcal{L}W_\lambda$ and $\mathcal{L}W_\sigma$ are small in the sense that
      there exists a universal constant $C_0$ which only depends on $K$ and $n$ such that for $(\lambda,\sigma)$ given in \eqref{eq.independent_parameters}, $\mu>0$ small enough, there hold
      \begin{equation}\label{eq.LW_lambda_sigma}
            \begin{split}
                  \sup_{f\in\mathcal{H}_{k,0}} \frac{\left|\left<f,\mathcal{L}W_\lambda\right>\right|}{\left\Vert f\right\Vert_{L^{\infty}\left(\mathbf{B}\right)}}
                  &\leq C_0\mu^{\frac{1}{2}}\left|\ln\mu\right|,\\
                  \sup_{f\in\mathcal{H}_{k,0}} \frac{ \left|\left<f,\mathcal{L}W_\sigma\right>\right|}{\left\Vert f\right\Vert_{L^\infty(\mathbf{B})}}
                  &\leq C_0\mu^{\frac{1}{2}}\left|\ln\mu\right|+C_0\left|L\left(\sigma\right)\right|.
            \end{split}
      \end{equation}
\end{lemma}
\begin{proof}
      Without loss of generality, we can assume that $\left\Vert f\right\Vert_{L^\infty(\mathbf{B})}=1$. For $t=\lambda$ or $t=\sigma$, by symmetry,
      \begin{equation*}
           \left<f,\mathcal{L}W_t \right>=-\int_{\mathbf{B}}f{\left(\Delta W+K(x)W^{5+\mu}\right)}_t dx=-\mu^{-\frac{1}{2}}\int_{\Omega_0}fF_t dx \leq C\int_{\Omega_0}\left|F_t\right|\mu^{-\frac{1}{2}}dx.
      \end{equation*}
      Note that
      \begin{equation*}
            \begin{aligned}
                  F_t & = \sum_{i=0}^{k-1}{\left[K(r)\left(r^2U_{i^*}^{5}-U_i^{5}\right)\right]}_t+(5+\mu)K(x)W^{4+\mu}\bigl[{(U_{0})}_{t}+R_t\bigr] \\
                  & =:  g_1+g_2+g_3+g_4+g_5,
            \end{aligned}
      \end{equation*}
      where
      \begin{equation*}
            \begin{aligned}
                  g_1 & = \sum_{i=0}^{k-1}\left[\left(r^2K'(r)+2rK(r)\right)U_{i^*}^5-K'(r)U_i^5\right]r_t , \\
                  g_2 & =\sum_{i=0}^{k-1}5K\left(r\right)r^2U_{i^*}^4(U_{i^*})_{t}-\sum_{i=1}^{k-1}5K(r)U_i^4 (U_{i}){t},  \\
                  g_3 & =\left(5+\mu\right)K\left(x\right)W^{4+\mu}R_t,\\
                  g_4 & =5\Bigl(K(x)\frac{(5+\mu)}{5}W^\mu-K(r)\Bigr)W^4 (U_{0})_{t},\\
                  g_5 & =5K\left(r\right)\left(W^4-U_0^4\right)(U_{0})_{t}.
            \end{aligned}
      \end{equation*}
      We estimate each term as follows.\vspace{5pt}
      
      Since $\partial_{\lambda}r=0$ and $\partial_{\sigma}r=-\mu^\frac{1}{2}$, we obtain
      \begin{equation*}
            \begin{split}
                  \int_{\Omega_0}|g_1|\mu^{-\frac{1}{2}}dx
                  &\leq C \mu^\frac{1}{2}\sum_{i=0}^{k-1}\int_{\Omega_0}U_i^5\varepsilon^{-\frac{1}{2}}dx\\
                  &\leq C\mu^\frac{1}{2}\sum_{i=0}^{k-1}\int_{\Omega_i}U_0^5\varepsilon^{-\frac{1}{2}}dx\\
                  &\leq C \mu^\frac{1}{2}\int_{\mathbb{R}^3}\Phi^5(y)dy=O\bigl(\mu^\frac{1}{2}\bigr).
            \end{split}
      \end{equation*}
      
      Since $\left|(U_{i^*})_{t}\right|=O(U_{i^*})$ for $x\in \mathbf{B}$; and $|(U_{i})_{t}|=O(U_{i^*})$ for $x\not\in \Omega_i$, by symmetry and \eqref{eq.d_i_d_i_star_estimates},
      \begin{equation*}
            \int_{\Omega_0}|g_2|\mu^{-\frac{1}{2}}dx \leq C\sum_{i=0}^{k-1}\int_{\Omega_0}\frac{1}{d_{i^*}^5}\frac{dx}{\varepsilon^3} \leq C\int_{\mathbf{B}}\frac{1}{d_{0^*}^5}\frac{dx}{\varepsilon^3}=O(\mu).
      \end{equation*}
      In the following, we will use \eqref{eq.d_i_d_i_star_estimates} frequently.
      Set $\theta=\frac{5}{8}$ and using \eqref{eq.W_d_theta_mu}, \eqref{eq.R_estimates_parameters}, 
      and $W^{\mu}=O(1)$ we get
      \begin{equation*}
            \begin{aligned}
                  \int_{\Omega_0}|g_3|\mu^{-\frac{1}{2}}dx ={}&
                  O(1)\times\int_{\Omega_0}\Bigl(\frac{\mu^{2(1-\theta)}}{d_{0^*}^{4\theta}}+\frac{1}{d_0^4} \Bigr)\left(\frac{1}{d_{0^*}}+\sum_{i=1}^{k-1}\frac{\mu^{-\frac{1}{2}}}{d_{i^*}^2}\right)\frac{dx}{\varepsilon^3} \\
                  ={}& O\bigl(\mu^{2(1-\theta)}\bigr)\times\Bigl(\frac{1}{d_{0^*}^{4\theta-2}(r\mathbf{m}_0)}+\sum_{i=1}^{k-1}\frac{\mu^{-\frac{1}{2}}}{d_{i^*}^{4\theta-1}(r\mathbf{m}_0)}\Bigr)\\
                  & +O(1)\times\Bigl(\frac{1}{d_{0^*}(r\mathbf{m}_0)}+\sum_{i=1}^{k-1}\frac{\mu^{-\frac{1}{2}}}{d_{i^*}^2(r\mathbf{m}_0)}\Bigr) \\
                  ={}&O(\mu)+O\bigl(\mu^\frac{1}{2}\bigr)=O\bigl(\mu^\frac{1}{2}\bigr).
            \end{aligned}
\end{equation*}
      
      Using $K(x)\frac{(5+\mu)}{5}W^\mu-K(r)=O(\mu d_0\left|\ln \mu\right|)$, \eqref{eq.W_d_theta_mu} with $\theta=\frac{3}{4}$ and $(U_{0})_{t}=\frac{O(U_0)}{\mu^\frac{1}{2}d_0}=O\left(\frac{1}{\mu d_0^2}\right)$ we obtain
      \begin{equation*}
            \int_{\Omega_0}\left|g_4\right|\mu^{-\frac{1}{2}}dx \leq C\int_{\Omega_0}\mu d_0\Bigl(\frac{1}{d_0^4}+\frac{\mu^\frac{1}{2}}{d_{0^*}^3}\Bigr)\frac{1}{\mu d_0^2}\frac{\left|\ln \mu\right|dx}{\varepsilon^3}=O\left(\mu^\frac{1}{2}\left|\ln \mu\right|\right).
      \end{equation*}
      
      Finally we estimate $g_5$. By consider case $U_0>|R|$ and $U_0\leq |R|$ we can show that
      \begin{equation*}
            W^4-U_0^4 \leq C\left(U_0^3\left|R\right|+R^4\right) \leq C{\left(\max\left\{U_0,|R|\right\}\right)}^3\left|R\right|.
      \end{equation*}
      Taking for $\theta=\frac{5}{8}$ in \eqref{eq.epsilon_R_theta_estimates} and using $(U_{0})_{t}=O\left(\frac{1}{\mu d_0^2}\right)$ we get
      \begin{equation*}
            \int_{\Omega_0}|R|^4|(U_{0})_{t}|\mu^{-\frac{1}{2}}dx=\int_{\Omega_0}\frac{O\bigl(\mu^\frac{1}{4}\bigr)}{d_{0^*}^\frac{5}{2}d_0^2}\frac{dx}{\varepsilon^3}=\frac{O\bigl(\mu^\frac{1}{4}\bigr)}{d_{0^*}^\frac{3}{2}(r\mathbf{m}_0)}=O(\mu).
      \end{equation*}
      (i) Suppose $t=\lambda$. Using $(U_{0})_{\lambda}=O(U_0)$ and \eqref{eq.R_estimates_parameters} we have
      \begin{equation*}
            \int_{\Omega_0}U_0^3|(U_{0})_{\lambda}R|\mu^{-\frac{1}{2}}dx \leq C\int_{\Omega_0}\frac{1}{d_0^4}\Bigl(\frac{1}{d_{0^*}}+\sum_{i=1}^{k-1}\frac{\mu^{-\frac{1}{2}}}{d_{i^*}^2}\Bigr)\frac{dx}{\varepsilon^3}=O\bigl(\mu^\frac{1}{2}\bigr).
      \end{equation*}
      (ii) Suppose $t=\sigma$. Using \eqref{eq.R_estimates_x}, we expand
      \begin{equation*}
            \begin{aligned}
                  U_0^3|R| |(U_{0})_{\sigma}| & \leq C\frac{1}{d_0^3}\Bigl[
                        {|\mu^{-2}R(r\mathbf{m}_0)|}+\| \mu^{-2}\nabla R\|_{L^\infty}\mu d_0
                  \Bigr]\frac{\mu^{-\frac{1}{2}}}{d_0^2} \\
        & \leq \frac{C\left|E(r\mathbf{m}_0)\right|}{d_0^5}+\frac{C\mu^\frac{1}{2}}{d_0^4}.
            \end{aligned}
      \end{equation*}
      Then we derive
      \begin{equation*}
            \begin{aligned}
                  \int_{\Omega_0}|g_5|\mu^{-\frac{1}{2}}dx
                   & \leq C|E(r\mathbf{m}_0)|\int_{\Omega_0}\frac{1}{d_0^5}\frac{dx}{\varepsilon^3}+C\mu^\frac{1}{2}\int_{\Omega_0}\frac{1}{d_0^4}\frac{dx}{\varepsilon^3}+O\bigl(\mu^\frac{1}{2}\bigr) \\
                   & \leq C\bigl(\left|E(r\mathbf{m}_0)\right|+\mu^\frac{1}{2}\bigr).
            \end{aligned}
      \end{equation*}
      Finally, applying the estimate of $|E(r\mathbf{m}_0)|$ in \eqref{eq.E_x_0_and_nabla_E_x_0}, the proof is finished.
\end{proof}

\subsection{The constrained linearized problem}

To solve \eqref{eq.c_1_c_2_equation}, we first investigate the existence of the constrained linear equation: given $f\in \mathcal{H}_k$,
finding $(c_1,c_2,\varphi)\in \mathbb{R}\times\mathbb{R}\times\mathcal{H}_{k,0}$
such that
\begin{equation}\label{eq3.1}
      \mathcal{L}\varphi=f+c_1\Delta W_\lambda+c_2\Delta W_\sigma ~\text{in}~ \mathbf{B}.
\end{equation}

\begin{proposition}\label{th.2}
      There exists a positive constant $\mu_0$ such that for each $0 < \mu \leq \mu_0$ and any $(\lambda,\sigma)$ in \eqref{eq.independent_parameters},
      problem (\ref{eq3.1}) with $f \in \mathcal{H}_k$ admits a unique solution. Moreover, for each $\rho \in \left[\frac{1}{2},1\right)$,
      there exists a constant $C(K,\rho)$ such that the solution satisfies
      \begin{equation}\label{eq3.2}
            \left\Vert \varphi \right\Vert_\rho +\mu^{-\frac{1}{2}}|c_1|+ \mu^{-1}|c_2|
            \leq C(K,\rho) \mu^{2}\left\Vert f\right\Vert_{\rho+2}.
      \end{equation}
\end{proposition}
\begin{proof}
      By Fredholm alternative, we only need to establish the priori estimate (\ref{eq3.2}).
      To this end, we assume that $(c_1,c_2,\varphi,f)\in \mathbb{R} \times \mathbb{R} \times \mathcal{H}_{k,0} \times \mathcal{H}_k$ is a quadruple satisfying (\ref{eq3.1}).
      \par
      Let $\rho \in \left[\frac{1}{2},1\right)$ be fixed. Note from \eqref{eq.omega_rho_omega_rho_star_estimates} that $\omega_\rho\leq C\omega_{1/2}^{2\rho}=O(1)$.\par
      We first estimate $c_1,~c_2$.  Multiplying (\ref{eq3.1}) by vector $\bigl(\lambda W_\lambda,\lambda^{-1}\mu^{-\frac{1}{2}}W_{\sigma}\bigr)^\intercal$ and integrating over $\mathbf{B}$, we obtain
      \begin{equation}
            \label{eq3.3}
            M \bbm \lambda^{-1}\mu^{-\frac{1}{2}}c_1 \vspace{8pt}\\ \lambda \mu^{-1} c_2 \ebm =
             \bbm \lambda \langle f,W_\lambda \rangle - \lambda \langle\varphi,\mathcal{L}W_\lambda \rangle  \vspace{10pt}\\  \lambda^{-1}\mu^\frac{1}{2} \langle f,W_\sigma \rangle - \lambda^{-1}\mu^{\frac{1}{2}} \langle \varphi, \mathcal{L}W_\sigma \rangle\ebm,
      \end{equation}
      where $M$ is defined in \eqref{eq.M_estimates}. Set $V_0=U_0-U_{0^*}$, then
      \begin{align*}
            \langle \varphi,\mathcal{L}W_t \rangle = k \langle \varphi,\mathcal{L}(V_{0})_{t} \rangle \thicksim \mu^{-\frac{1}{2}} \langle \varphi,\mathcal{L}(V_{0})_{t} \rangle.
      \end{align*}Hence, using 
      \begin{align*}
            |(V_{0})_{\lambda}(x)|+\mu^\frac{1}{2}|(V_{0})_{\sigma}(x)|\leq C U_0(x)\leq C \varepsilon^\frac{1}{2}|x-r\mathbf{m}_0|^{-1},
      \end{align*}
      $|f|\leq \Vert f \Vert_{5/2} \omega_{5/2}$ and \eqref{eq.rho_estimates_in_B_omega}, we see that
      \begin{equation*}
            |\langle f,W_\lambda \rangle|+\mu^\frac{1}{2}|\langle f,W_\sigma \rangle| 
            \leq C \mu^2\left\Vert f \right\Vert_\frac{5}{2} \int_{\mathbf{B}}\frac{\omega_\frac{5}{2}(x)}{|x-r\mathbf{m}_0|}\frac{dx}{\varepsilon^3}\leq C \mu^\frac{5}{2}\left\Vert f \right\Vert_\frac{5}{2}.
      \end{equation*}
      Using \eqref{eq.LW_lambda_sigma} and \eqref{eq.f_norm_estimates}, we can claim 
      \begin{equation*}
          |\langle \varphi,\mathcal{L}W_\lambda \rangle|+\mu^\frac{1}{2}|\langle \varphi,\mathcal{L}W_\sigma \rangle| 
            \leq C \mu^\frac{1}{2}\left|\ln \mu\right| \left\Vert \varphi \right\Vert_{L^\infty}
            \leq C \mu^\frac{1}{2}\left|\ln \mu\right| \left\Vert \varphi \right\Vert_\frac{1}{2}.
      \end{equation*}
      Since $M$ in \eqref{eq.M_estimates} has a bounded inverse, we obtain
      \begin{equation}
            \label{eq3.4}
            |c_1|+ \mu^{-\frac{1}{2}}|c_2| \leq C\mu\left|\ln \mu\right| \left\Vert \varphi \right\Vert_\frac{1}{2}+C\mu^\frac{5}{2}\left\Vert f \right\Vert_\frac{5}{2}.
      \end{equation}\par
      Next, we estimate $\varphi$. Using $-\Delta \varphi = (5+\mu)K(x)W^{4+\mu}\varphi + f + c_1 \Delta W_\lambda + c_2 \Delta W_\sigma$, 
      \eqref{eq.inverse_laplacian_estimates}, and Lemma \ref{lemma.W_norm_estimates},
      we obtain, for $2 \theta = \min \{ 1,1-\rho \}$,
      \begin{align}\label{eq.temp_varphi_righthandside}
      \begin{aligned}
      \left|\varphi(x)\right|  \leq & C\mu^2\bigl(\left\Vert W^{4+\mu}\varphi \right\Vert_{\rho+2+\theta}+\left\Vert \Delta W_\lambda \right\Vert_5 |c_1|+\left\Vert \Delta W_\sigma \right\Vert_6 |c_2|\bigr)\omega_{\rho+\theta}(x) \\
                  & +C\mu^2\left\Vert  f \right\Vert_{\rho+2}\omega_\rho (x)   \\
                 \leq {}& C\bigl(\left\Vert\varphi\right\Vert_\rho + \mu^{-\frac{1}{2}}|c_1|+ \mu^{-1}|c_2|\bigr)\omega_{\rho+\theta}(x)+C\mu^2 \left\Vert f \right\Vert_{\rho+2}\omega_\rho (x).
            \end{aligned}
      \end{align}
      Substituting (\ref{eq3.4}) into the right-hand side of \eqref{eq.temp_varphi_righthandside} together with $\omega_{\rho+\theta}<\omega_\rho$ and \eqref{eq.f_norm_estimates}, we derive
      \begin{equation}
            \label{eq3.5}
            \frac{\left|\varphi (x)\right|}{\omega_\rho (x)}\leq
            C \Bigl(\left\Vert \varphi \right\Vert_\rho \frac{\omega_{\rho+\theta}(x)}{\omega_\rho (x)}+\mu^2 \left\Vert  f \right\Vert_{\rho+2} \Bigr), \quad ~\text{for all}~ x\in \overline{\Omega}_0.
      \end{equation}
      Now, we will finish the proof by contradiction. Suppose the priori estimate (\ref{eq3.2}) is not true. Then along a sequence of integer $m \rightarrow +\infty$ combined with $\mu_m :=  \frac{1}{m^2} \rightarrow 0 $, there exist $(\lambda_m,\sigma_m)$ satisfying \eqref{eq.independent_parameters} and $(c_{1,m},c_{2,m},\varphi_m,f_m)\in \mathbb{R}^2\times \mathcal{H}_{k,0}\times\mathcal{H}_k$ such that
      \[\max\left\{\varepsilon_m^\frac{1}{2}\left\Vert\varphi_m \right\Vert_\rho,|c_{1,m}|+\lambda_m \mu_m^{-\frac{1}{2}}|c_{2,m}| \right\}=1,\]
       and 
       \[\varepsilon^\frac{5}{2}_m\left\Vert  f_m \right\Vert_{\rho+2}\rightarrow 0, \quad \text{as}~ m\rightarrow +\infty.\]
       
      By (\ref{eq3.4}),
      $|c_{1,m}|+\lambda_m \mu_m^{-\frac{1}{2}}|c_{2,m}|\rightarrow 0$ as $m\rightarrow +\infty$. Hence, $\varepsilon^\frac{1}{2}_m\left\Vert  \varphi_m \right\Vert_\rho =1$. Consequently,
      \[\varepsilon^\frac{1}{2}_m\left\Vert  \varphi_m \right\Vert_{L^\infty}\leq \left\Vert \omega_\rho \right\Vert_{L^\infty}=O(1).\]
      \par
      Let $z_m\in \overline{\Omega}_0$ be a point such that 
      \[1=\varepsilon^\frac{5}{2}_m\left\Vert  f_m \right\Vert_{\rho+2}=\frac{\left|\varepsilon^\frac{1}{2}_m \varphi_m (z_m)\right|}{\omega_\rho (z_m)}.\]
      By (\ref{eq3.5}), $\omega_\rho (z_m)\leq C\omega_{\rho+\theta}(z_m)$.
      Since $\omega_{\rho+\theta}\leq \Bigl(\min_i {\bigl[1+\frac{|x-r\mathbf{m}_i|}{\varepsilon}\bigr]}^{-\theta}\Bigr)\omega_\rho$, there is a constant $C_0>0$ such that 
      \[\frac{|z_m-r\mathbf{m}_0|}{\varepsilon_m} \leq C_0.\] 
      \par
      Now, we set $u_m(z)=\varepsilon_m^\frac{1}{2}\varphi_m (r\mathbf{m}_0+\varepsilon_m z)$, while $\mathcal{L}_m\varphi_m=f_m + c_{1,m}\Delta W_\lambda + c_{2,m}\Delta W_\sigma$.
      Then using \eqref{eq.R_estimates_parameters} we have, when $|z|<\lambda_m \mu_m^{-\frac{1}{2}}\sigma$,
      \begin{equation*}
            \begin{aligned}
                  - & \Delta_z u_m - (5+\mu_m){\left(\Phi(z)+O\left(\mu_{m}^\frac{1}{2}\right)\right)}^{4+\mu_m}u_m
                  \\[5pt]
                    & = \varepsilon_m^\frac{5}{2}f_m + c_{1,m}\varepsilon_m^\frac{5}{2}\Delta W_\lambda + \lambda_m \mu^{-\frac{1}{2}}c_{2,m} \cdot\varepsilon_m^{\frac{5}{2}}\lambda_m^{-1}\mu^\frac{1}{2}\Delta W_\sigma.
            \end{aligned}
      \end{equation*}
      \par
      Since $\Vert \cdot \Vert_{L^\infty}\leq C \Vert \cdot \Vert_\tau, ~\text{for any}~ \tau \geq \frac{1}{2}$ and up to a subsequence, $u_m$ converges uniformly in any compact subset of $\mathbb{R}^3$ to a bounded solution $u$ of
      \begin{equation}\label{eq3.6}
            -\Delta u(z)-5\Phi^4(z)u(z) = 0 ~\text{in}~ \mathbb{R}^3,\quad\frac{1}{{\left(1+C_0^2\right)}^\frac{\rho}{2}}\leq \max_{|z|\leq C_0} |u(z)|.
      \end{equation}
      \par
       By elliptic theory, we can prove that $u \in D^{1,2}(\mathbb{R}^3)$ and $\|\nabla u\|_{L^{\infty}(\mathbb{R}^3)} = O(|x|^{-2})$.      
      Then, by \cite[Theorem 2.1]{Bartsch_Tobias_Willem_2003}, there are constants $a_0,~a_1,~a_2,~a_3$ such that 
      \[u(z)=a_0\Bigl[\frac{1}{2}\Phi(z)+z\cdot \nabla \Phi(z)\Bigr]+\sum_{i=1}^{3}a_i \Phi_{z_i}(z).\]
      Since $u(z)$ is even with respect to $z_2,~z_3,~a_i=0$ for $i=2,3$. Next, since $\varphi\in\mathcal{H}_{k,0}$ we can use Lebesuge dominated convergence theorem to derive that
      \begin{equation*}
            \int_{\mathbb{R}^3}\Phi^4(z)\Bigl[ \frac{1}{2}\Phi(z)+z\cdot \nabla \Phi(z)\Bigr]u(z)dz=0,\quad\int_{\mathbb{R}^3}\Phi^4(z)\Phi_{z_1}(z)u(z)dz=0.
      \end{equation*}
      This implies that $u=0$, contradicting (\ref{eq3.6}). This contradiction completes the proof.
\end{proof}

\section{Reduction and Existence of Critical Points}\label{sec.3}

\subsection{The constrained nonlinear problem}

Recall that the constrained nonlinear problem \eqref{eq.c_1_c_2_equation} is
\begin{equation*}
      (c_1,c_2,\varphi)\in \mathbb{R}\times \mathbb{R}\times \mathcal{H}_{k,0}, \quad \mathcal{L}\varphi= F+N(\varphi)+c_1\Delta W_\lambda +c_2\Delta W_\sigma ,
\end{equation*}

Here we will apply Proposition \ref{th.2} and the contraction mapping theorem to solve \eqref{eq.c_1_c_2_equation}. 
First we consider the following nonlinear function in \eqref{eq.linear_error}.
\begin{align*}
      N(\varphi)=K(x)\left[\left|W+\varphi\right|^{5+\mu}-W^{5+\mu}-(5+\mu)W^{4+\mu}\varphi\right].
\end{align*}
\begin{lemma}\label{lemma3.1}
      Let $\rho\geq \frac{1}{2}$,
      then for every $\varphi,\psi \in \mathcal{H}_k$, we have
\begin{align}\label{eq3.7}
      \frac{\mu^2\left\Vert N(\varphi)-N(\psi)\right\Vert_{\rho+2}}{\left\Vert\varphi-\psi\right\Vert_{\rho}}  \leq C \mathcal{M}_1,
\end{align}
\begin{align}\label{eq3.8}
      \mu^2\left\Vert N(\varphi)\right\Vert_{\rho+1}\leq C \mathcal{M}_2,
\end{align}
where
\begin{align*}
      \mathcal{M}_1 & =\max\left\{\mu^\frac{1}{2}\left\Vert \varphi\right\Vert_{1},\mu^\frac{1}{2}\left\Vert \psi\right\Vert_{1}
      ,\mu^2\left\Vert\varphi\right\Vert_{1}^{4+\mu},\mu^2\left\Vert\psi\right\Vert_{1}^{4+\mu} \right\},
      \\ \mathcal{M}_2&=\max\left\{\mu^\frac{1}{2}\left\Vert\varphi\right\Vert_{\frac{3}{4}},
      \mu^2\left\Vert\varphi\right\Vert_{\frac{3}{4}}^{4+\mu} \right\}\left\Vert\varphi\right\Vert_\rho.
           \\
\end{align*}

\end{lemma}
\begin{proof}
      By the interpolation \eqref{eq.omega_rho_omega_rho_star_estimates},
      \begin{equation*}
            \omega_{1}^{4}\omega_\rho \leq C\omega_{\rho+2}^{\frac{2\rho-1+4}{2\rho+3}} \leq C \omega_{\rho+2}.
      \end{equation*}
      \par
      Then, we have
      \begin{equation*}
            \begin{aligned}
                   & \frac{\mu^2|N(\varphi)-N(\psi)|}{\Vert \varphi-\psi \Vert_\rho}     \\
                   & \quad \leq C \mu^2\max {\left\{W,\left|\varphi\right|,\left|\psi\right|\right\}}^{3+\mu}\left(\left|\varphi\right|+\left|\psi\right|\right)\frac{|\varphi-\psi|}{\Vert\varphi-\psi \Vert_\rho}       \\
                   & \quad \leq C \max {\left\{ \mu^\frac{1}{2}\left\Vert W \right\Vert_{1},\mu^\frac{1}{2}\left\Vert \varphi \right\Vert_{1},\mu^\frac{1}{2} \left\Vert \psi \right\Vert_{1} \right\}}^{3+\mu}\mu^\frac{1}{2}\left(\left\Vert \varphi \right\Vert_{1}+ \left\Vert \psi \right\Vert_{1}\right) \omega_{1}^{4}\omega_\rho, \\
                   & \quad \leq C \max {\left\{ \mu^\frac{1}{2}\left\Vert W \right\Vert_{1},\mu^\frac{1}{2}\left\Vert \varphi \right\Vert_{1},\mu^\frac{1}{2} \left\Vert \psi \right\Vert_{1} \right\}}^{3+\mu}\mu^\frac{1}{2}\left(\left\Vert \varphi \right\Vert_{1}+ \left\Vert \psi \right\Vert_{1}\right)\omega_{\rho+2}.
            \end{aligned}
      \end{equation*}
      \par
      Since $\mu^\frac{1}{2}\left\Vert W \right\Vert_{1}$ is bounded, (\ref{eq3.7}) holds. The proof of (\ref{eq3.8}) is analogous. This completes the proof.
\end{proof}

Now we are in position to solve the nonlinear problem of \eqref{eq.c_1_c_2_equation}.

\begin{proposition}\label{th.3}
      There exists a small $\mu_0$ such that for every $\mu\in (0, \mu_0]$ and $(\lambda,\sigma)$ in \eqref{eq.independent_parameters}, \eqref{eq.c_1_c_2_equation}
      admits a unique solution $(c_1,c_2,\varphi)$ satisfying
      \begin{gather}
            \label{eq3.10}
            \mu^{-\frac{1}{2}}|c_1|+\mu^{-1}|c_2|+\left\Vert\varphi\right\Vert_\rho
            \leq C\bigl(|L\left(\sigma\right)|+\mu^{\frac{1}{2}}\left|\ln \mu\right| \bigr), ~\text{for all}~\frac{1}{2}\leq \rho < 1,\\[2pt]
            \label{eq3.11}
           \left\Vert \varphi\right\Vert_1
            \leq C\bigl( |L\left(\sigma\right)|+\mu^{\frac{1}{2}}|\ln^2\mu|\bigr).
      \end{gather}
\end{proposition}
\begin{proof}
We will divide our proof into three steps.

\smallskip
Step 1. Let $\beta=\frac{8}{9}$, $X=\{ f\in\mathcal{H}_k:\left\Vert f \right\Vert_\beta \leq \mu^{-\frac{1}{10}} \}$ and let $\varphi\in X$. By \eqref{eq.f_norm_estimates}, we see that
      \begin{equation}
            \label{eq3.12}
            \left\Vert \varphi \right\Vert_{1}\leq (2\lambda)^{\frac{1}{9}} \mu^{-\frac{1}{9}}\left\Vert \varphi \right\Vert_{\beta}\leq C \mu^{-\frac{1}{10}-\frac{1}{9}}.
      \end{equation}
     Using \eqref{eq.F_estimates} to $F$ and (\ref{eq3.7}) to $N(\varphi)$, we find that
      \begin{equation*}
            \left\Vert F+N(\varphi) \right\Vert_{\beta+2}\leq C \left[\mu^{-2}\left|\ln \mu\right|+\mu^{-\frac{1}{5}-\frac{2}{9}-\frac{3}{2}}\right]=O\left(\mu^{-2}\left|\ln \mu\right|\right).
      \end{equation*}
      Define $(c_1,c_2,\psi)$ in $\mathbb{R}^2\times \mathcal{H}_{k,0}$ as the unique solution of $\mathcal{L}\psi =F+N(\varphi)+c_1\Delta W_\lambda+c_2 \Delta W_\sigma$, given by Proposition \ref{th.2} and
      \begin{equation*}
            \left\Vert \psi \right\Vert_\beta \leq C\left\Vert \mu^{2} 
            \left(F+N(\varphi)\right) \right\Vert_{\beta+2}\leq C\left|\ln \mu\right| \leq \mu^{-\frac{1}{10}},~\text{for}~ \mu ~\text{small}.
      \end{equation*}
       We now define $T:X\to X$ with $\varphi\mapsto \psi$ for $\mu>0$ small. In addition, for any $\varphi_1,\  \varphi_2\in X$, by (\ref{eq3.7}) and (\ref{eq3.12}) we have
      \begin{equation*}
            \left\Vert T(\varphi_1)-T(\varphi_2) \right\Vert_\beta \leq C\mu^2\left\Vert N(\varphi_1)-N(\varphi_2) \right\Vert_{\beta+2}=o(1)\left\Vert  (\varphi_1-\varphi_2) \right\Vert_\beta.
      \end{equation*}
      \par
      Hence, $T$ is a contraction operator for $\mu>0$ small. Consequently, when $\mu>0$ small, the contraction mapping theorem ensures that there exists a unique fixed point of $T$ in $X$, which gives a unique solution, denoted by $(c_1,c_2,\varphi)$, of \eqref{eq.c_1_c_2_equation} in $X$.

      \smallskip Step 2. We apply (\ref{eq3.2}) to $\mathcal{L}\varphi=F+N(\varphi)+c_1\Delta W_\lambda+c_2 \Delta W_\sigma$ with $\rho=\frac{1}{2}$ to obtain
      \begin{equation*}
            \left\Vert \varphi \right\Vert_{\frac{1}{2}}+\mu^{-\frac{1}{2}}|c_1|+\mu^{-1}|c_2| \leq C \mu^2 \bigl(\left\Vert F \right\Vert_\frac{5}{2}+\left\Vert N(\varphi) \right\Vert_{\frac{5}{2}}\bigr).
      \end{equation*}
      \par
      By (\ref{eq3.7}) and (\ref{eq3.12}), $\left\Vert N(\varphi) \right\Vert_\frac{5}{2}=o(1)\mu^{-2}\left\Vert \varphi \right\Vert_\frac{1}{2}$. Thus by \eqref{eq.F_estimates}, we obtain (\ref{eq3.10}) with $\rho=\frac{1}{2}$:
      \begin{equation}\label{eq.estimate_temp_frac_1_2}
            \left\Vert \varphi \right\Vert_{\frac{1}{2}}+\mu^{-\frac{1}{2}}|c_1|+\mu^{-1}|c_2| \leq C \bigl(|L\left(\sigma\right)|+\mu^{\frac{1}{2}}\left|\ln \mu\right| \bigr).
      \end{equation}

      \smallskip Step 3. We notice from \eqref{eq.F_estimates} that there exist $G_1$, $G_2$, such that $F=G_1+G_2$ and
      \[\begin{aligned}
            &\left\Vert G_1 \right\Vert_4 \leq C \left(L\left(\sigma\right)\mu^{-2}\right)z,\quad \left\Vert  G_2 \right\Vert_{\rho+2} \leq C \mu^{\frac{\min\{1,2-\rho \}}{2}-2}\left|\ln \mu\right|,~\text{for}~ \rho\in[0,3).
            \end{aligned}\]
      Hence, from
      \begin{equation*}
            -\Delta \varphi=\left(5+\mu\right) K(x) W^{4+\mu}\varphi + F_1 + F_2 + N(\varphi)+ c_1\Delta W_\lambda + c_2\Delta W_\sigma,
      \end{equation*}
      With \eqref{eq.f_norm_estimates} and \eqref{eq.inverse_laplacian_estimates}, we obtain, for any $\rho\in\left(\frac{1}{2},1\right]$,
      \begin{equation*}
            \begin{aligned}
                 \left\Vert \varphi\right\Vert_\rho \leq & C \mu^2\Big(\left\Vert F_1\right\Vert_4+\left\Vert F_2 \right\Vert_{\rho+2}(1+\mathbbold{1}_{\{\rho=1\}}\left|\ln \mu\right|)                                                                                          \\
                & +\left\Vert \varphi W^{4+\mu} \right\Vert_{\rho+\frac{5}{2}}+|c_1|\left\Vert \Delta W_\lambda \right\Vert_4+|c_2|\left\Vert \Delta W_\sigma \right\Vert_4+\left\Vert N(\varphi) \right\Vert_{\rho+2+\eta_{\rho}}\Big)                                          \\
                \leq & C\Big( \left|L\left(\sigma\right)\right|+\mu^{\frac{1}{2}}\left|\ln \mu\right|\left(1+\mathbbold{1}_{\{\rho=1\}}\left|\ln \mu\right|\right) \Big)                           \\
                & +C \mu^2 \Big(\left\Vert \varphi W^{4+\mu} \right\Vert_{\frac{1}{2}+3}+|c_1|\left\Vert \Delta W_\lambda \right\Vert_5+|c_2|\left\Vert \Delta W_\sigma \right\Vert_6+\left\Vert N(\varphi) \right\Vert_{\rho+2+\eta_{\rho}}\Big),                                 
            \end{aligned}
      \end{equation*}
      where $\eta_{\rho}>0$ is an appropriate constant.
      Then, by Lemma \ref{lemma.W_norm_estimates} and (\ref{eq3.7}), \eqref{eq.estimate_temp_frac_1_2}, we have 
      \begin{equation*}
            \begin{aligned}
                \left\Vert \varphi\right\Vert_\rho \leq & C\left( \left|L\left(\sigma\right)\right|+\mu^{\frac{1}{2}}\left|\ln \mu\right|\left(1+\mathbbold{1}_{\{\rho=1\}}\left|\ln \mu\right|\right) \right)                                              \\
                & +C\left( \Vert \varphi \Vert_\frac{1}{2}+\mu^{-\frac{1}{2}}|c_1|+\mu^{-1}|c_2| \right)+C \left\Vert\varphi \right\Vert_1 \left(\mu^{\frac{1}{2}-\eta_{\rho}}\left\Vert \varphi \right\Vert_{\rho}\right)\\
               \leq & C\Big( \left|L\left(\sigma\right)\right|+\mu^{\frac{1}{2}}\left|\ln \mu\right|\left(1+\mathbbold{1}_{\{\rho=1\}}\left|\ln \mu\right|\right)\Big) +\left(C_*  \mu^{\frac{1}{2}-\frac{1}{10}-\frac{1}{9}-\eta_{\rho}}\right)\left\Vert \varphi \right\Vert_{\rho}.
            \end{aligned}
          \end{equation*}
     We notice that $C_*>0$ depends only on $\rho$. Hence, by (\ref{eq3.12}), there exist $\eta_{\rho}\in\left(0,\frac{1}{2}\right)$ and $\mu_0 >0$ small such that when $0<\mu \leq \mu_0$, we have $C_*\mu^{\frac{1}{2}-\frac{1}{10}-\frac{1}{9}-\eta_{\rho}}\leq \frac{1}{2}$. Since \eqref{eq.estimate_temp_frac_1_2} holds, this gives \eqref{eq3.10} for $\frac{1}{2}\leq\rho<1$ and \eqref{eq3.11} for $\rho = 1$. This completes the proof.\end{proof}
\subsection{Proof of Theorem \bref{th.1}}

To begin the proof, we need some estimates of energy functional $J$ with respect to parameters $\lambda$, $\sigma$ and $\mu$, whose proof is postponed in Appendix \ref{appendix.energy_estimates}.
\begin{lemma}\label{lemma.energy_estimates}
      Let 
      \begin{align*}
            J(\lambda,\sigma,\mu) := \int_{\mathbf{B}}\Bigl(\frac{1}{2}{\left|\nabla W\right|}^2-\frac{K(x)}{6+\mu}W^{6+\mu} \Bigr).
      \end{align*}
      Then for $(\lambda,\sigma)$ given in \eqref{eq.independent_parameters}
      \begin{equation*}
            \begin{split}
                  \partial_\lambda J(\lambda,\sigma,\mu)&=\lambda^{-2}A_2L\left(\sigma\right)+O\bigl(\mu^\frac{1}{2}\left|\ln\mu\right|\bigr),\\
                  \partial_\sigma J(\lambda,\sigma,\mu)&=K'(1) A_1-\lambda^{-1}A_2L'\left(\sigma\right)+O\bigl(\mu^{\frac{1}{2}}\left|\ln\mu\right|\bigr),\\
                  \partial_{\lambda^2}^2 J(\lambda,\sigma,\mu)&=-2\lambda^{-3}A_2L\left(\sigma\right)+O\bigl(\mu^\frac{1}{2}\left|\ln\mu\right|\bigr),\\
                  \partial_{\sigma\lambda}^2 J(\lambda,\sigma,\mu)&=\lambda^{-2}A_2L'\left(\sigma\right)+O\bigl(\mu^{\frac{1}{2}}\left|\ln\mu\right|\bigr).
            \end{split}
      \end{equation*}
\end{lemma}

\begin{proof}[Proof of Theorem \ref{th.1}]
Given $(\lambda,\sigma,\mu)$ satisfying \eqref{eq.independent_parameters} and $0<\mu\leq\mu_0$. Let $(c_1,c_2,\varphi)$ be the solution of \eqref{eq.c_1_c_2_equation} given by
Proposition \ref{th.3}. Now, the key problem is to find $(\lambda,\sigma)$ such that $(c_1,c_2)=(0,0)$. The pair $(c_1,c_2)$ satisfies the following system:
\begin{equation}
      \label{eq3.13}
      M \bbm \lambda^{-1}c_1\vspace{5pt}\\ \lambda\mu^{-\frac{1}{2}} c_2 \ebm = \bbm \lambda\mu^\frac{1}{2} & 0 \vspace{10pt}\\ 0 & \lambda^{-1}\mu \ebm \times
      \bbm  \left< F,W_\lambda\right> +\left<N(\varphi),W_\lambda \right> - \left< \varphi,\mathcal{L}W_\lambda \right>
      \vspace{10pt}\\ \left< F,W_\sigma\right> + \left<N(\varphi),W_\sigma \right> - \left< \varphi,\mathcal{L}W_\sigma \right> \ebm,
\end{equation}
where $M$ is given in Lemma \ref{lemma.M_estimates}. It is deduced by
multiplying \eqref{eq.c_1_c_2_equation} with $(\lambda W_\lambda,\lambda^{-1}\mu^{\frac{1}{2}}W_\sigma)^{\intercal}$ and using integration by parts.
Our aim is to demonstrate that the right-hand side of \eqref{eq3.13} vanishes at some point $(\lambda,\sigma)$ satisfying \eqref{eq.independent_parameters}. 

By direct calculation, we have
\begin{equation*}
      \langle F,W_\lambda \rangle=-\partial_\lambda J(\lambda,\sigma,\mu),\quad \langle F,W_\sigma \rangle=-\partial_\sigma J(\lambda,\sigma,\mu).
\end{equation*}

From (\ref{eq3.10}) we obtain
\begin{equation*}
      \left\Vert\varphi\right\Vert_{L^\infty(\mathbf{B})}
      \leq C\left\Vert\varphi \right\Vert_\frac{1}{2}
      \leq C\bigl( \left|L\left(\sigma\right)\right|+\mu^\frac{1}{2}\left| \ln \mu\right| \bigr).
\end{equation*}

Hence, using Lemma \ref{lemma.LW_lambda_sigma},
\begin{equation*}
      \begin{aligned}
            \langle \varphi,\mathcal{L}W_\lambda \rangle
             & \leq C\mu^\frac{1}{2}\left|\ln \mu\right|\left\Vert\varphi\right\Vert_{L^\infty}=O\left(\mu^\frac{1}{2}\left|\ln \mu\right|\right)
            \\
            \langle \varphi,\mathcal{L}W_\sigma \rangle
             & \leq C{\bigl(\left|L\left(\sigma\right)\right|+\mu^\frac{1}{2}\left|\ln\mu\right|\bigr)}^2
            \leq C L^2\left(\sigma\right)+C\mu^\frac{1}{2}\left|\ln \mu\right|.
      \end{aligned}
\end{equation*}

Since $|{\left(U_0-U_{0^*}\right)}_\lambda|\leq C \mu^\frac{1}{2}|x-r\mathbf{m}_0|^{-1}$ and $|{\left(U_0-U_{0^*}\right)}_\sigma|\leq C\mu\left|x-r\mathbf{m}_0\right|^{-2}$, for all fixed $\rho\in\left[\frac{1}{2},1\right)$, we can use \eqref{eq.rho_estimates_in_B_omega} to claim
\begin{equation*}
      \begin{aligned}
            |\langle N(\varphi),W_\lambda \rangle| & =k |\langle N(\varphi),{(U_0-U_{0^*})}_\lambda \rangle| \\
                               & \leq C\int_{\mathbf{B}}\frac{|N(\varphi)|}{|x-r\mathbf{m}_0|}dx
                        \\
                               & \leq C\left\Vert \mu^2\right\Vert_{\rho+2}
            \int_{\mathbf{B}}\frac{\omega_{\rho+2}(x)}{|x-r\mathbf{m}_0|}\frac{dx}{\varepsilon^2}                \\
            & \leq C\left\Vert \mu^2 \right\Vert_{\rho+2}\omega_\rho(r\mathbf{m}_0)              \\
            & \leq C\left\Vert \varphi \right\Vert_\rho\max\left\{  \mu^\frac{1}{2}\left\Vert\varphi \right\Vert_{1},{\left(\mu^2\left\Vert\varphi \right\Vert_{1}\right)}^{4+\mu}\right\}                  \\
            & \leq C\mu\left|\ln\mu\right|.
      \end{aligned}
\end{equation*}
Similarly, we have \[|\langle N(\varphi),W_\sigma \rangle|\leq C \mu^{-1} \bigl\Vert \mu^\frac{1}{2}\varphi \bigr\Vert_\rho\max\Bigl\{  \mu^\frac{1}{2}\left\Vert\varphi \right\Vert_{\frac{3}{4}},{\bigl(\mu^2\left\Vert\varphi \right\Vert_{\frac{3}{4}}\bigr)}^{4+\mu}\Bigr\}\leq C.\]
By Proposition \ref{th.3}, we have
\begin{equation*}
      \begin{aligned}
            \mu^\frac{1}{2} \left\Vert\varphi\right\Vert_{\rho} & \leq C \mu^\frac{1}{2}\bigl(\left|L\left(\sigma\right)\right|+\mu^\frac{1}{2}\left|\ln\mu\right|\bigr), ~\text{for any}~ \rho\in[\frac{1}{2},1),
            \\
            \mu^\frac{1}{2}\left\Vert\varphi\right\Vert_1      & \leq C\mu^\frac{1}{2}\bigl(\left|L\left(\sigma\right)\right|+\mu^\frac{1}{2}\left|\ln^2\mu\right|\bigr),
      \end{aligned}
\end{equation*}
which give the explicit bounds for $\mu^\frac{1}{2}\left\Vert\varphi\right\Vert_{\frac{1}{2}}$, $\mu^\frac{1}{2}\left\Vert\varphi\right\Vert_{\frac{3}{4}}$ and $\mu^\frac{1}{2}\left\Vert\varphi\right\Vert_{1}$. Then we obtain
\begin{equation*}
      \langle N(\varphi),W_\lambda \rangle =O\bigl(\mu^\frac{1}{2}\bigr),\hspace{2em}
      \langle N(\varphi),W_\lambda \rangle =O\bigl( L^2\left(\sigma\right)+\mu^\frac{1}{2}\left|\ln\mu\right| \bigr).
\end{equation*}
In conclusion, (\ref{eq3.13}) can be written as
\begin{equation}
      \label{eq3.14}
      \begin{aligned}
            M \bbm \lambda^{-1}c_1 \vspace{8pt}\\ \lambda\mu^{-\frac{1}{2}} c_2 \ebm
            = \bbm \mu^\frac{1}{2} & 0  \vspace{10pt}\\ 0 & \lambda^{-2}\mu \ebm T,
      \end{aligned}
\end{equation}
where
\begin{equation*}
      \begin{aligned}
            T = \bbm  -\lambda^{-1}A_2L\left(\sigma\right)+O\bigl(\mu^\frac{1}{2}\left|\ln\mu\right|\bigr)
            \vspace{10pt}\\ -K'\left(1\right) A_1\lambda + A_2L'\left(\sigma\right)+O\bigl( L^2(\sigma)+\mu^\frac{1}{2}\left|\ln\mu\right| \bigr) \ebm. 
      \end{aligned}
\end{equation*}
Since $L(\sigma^*)=0$ and $L'(\sigma^*)>0$, 
we can define a rectangle as follows \[Q :=  \bigl[\lambda^*-\mu^{\frac{1}{6}},\lambda^*+\mu^{\frac{1}{6}}\bigr]
      \times \bigl[\sigma^*-\mu^{\frac{1}{6}},\sigma^*+\mu^{\frac{1}{6}}\bigr].\]
Then, by degree theory, there exists one point in $Q$ at which the right-hand
side of (\ref{eq3.14}) vanishes. Denote one of such points by $\bigl(\hat{\lambda},\hat{\sigma}\bigr)$.
We have $L\left(\hat{\sigma}\right)=O\bigl(\mu^\frac{1}{2}\left|\ln\mu\right|\bigr)$ so $\hat{\sigma}=\sigma^*+O\bigl(\mu^\frac{1}{2}\left|\ln\mu\right|\bigr)$. Then, we have $K'(1)A_1 \hat{\lambda}= A_2 L'(\sigma^*)+O\bigl(\mu^{\frac{1}{2}}\left|\ln \mu\right|\bigr)$ so $\hat{\lambda}=\lambda^*+O\bigl(\mu^{\frac{1}{2}}\left|\ln \mu\right| \bigr)$.
Thus, there exists at least one pair \[\bigl(\hat{\lambda},\hat{\sigma}\bigr)=\left(\lambda^*,\sigma^*\right)+O\bigl(\mu^\frac{1}{2}\left|\ln\mu\right|\bigr)\] at which $c_1=c_2=0$. With $u=W+\varphi$, it provides a solution of $-\Delta u=K(x){\left|u\right|}^{5+\mu}$, so by the maximum principle, $u>0$ in $\mathbf{B}$. Hence, $u$ is a solution of \eqref{eq.basic_equation_n_3}.
\end{proof}

\begin{remark}
From the discussion above, notice that $u=U_0+R+\varphi$ in $\Omega_0$ and
\[\left\Vert R \right\Vert_{L^\infty(\Omega_0)}+\left\Vert\varphi \right\Vert_{L^\infty(\mathbf{B})}\leq C.\]
Under the assumption as Theorem \ref{th.1}, there exists a constant $\mu_0>0$ such that for all $\mu\in(0,\mu_0)$, problem \eqref{eq.basic_equation_n_3} admits a solution of the form
\begin{equation*}
      u(x)=\max_{0\leq i\leq k-1}\varepsilon^{-\frac{1}{2}}\Phi\Bigl(\frac{x-r\mathbf{m}_i}{\varepsilon}\Bigr)+O(1).
\end{equation*}
\end{remark}

\section{Acknowledgements}

Y. Pu is partially supported by the China Scholarship Council (No. 202306140126). He is grateful to Liping Wang for suggesting the problem and for her helpful discussions, and to Dong Ye for his valuable guidance and comments.

\appendix

\section{Some Basic estimates}\label{appendix}

We recall the following estimates in \cite{Hao2015} with the notations defined in Section \ref{subsec.notations}.

\begin{lemma}\label{lemma.rho_and_integrals}
      For each $\ell\in (0,3)$ and $\rho > 3-\ell$, we have
\begin{equation}
      \label{eq.rho_estimates_in_R}
      \int_{\mathbb{R}^3}\frac{1}{|X-Y|^\ell {(1+|Y|)}^\rho}dY=O(1)\times
      \left\{\begin{array}{ll}
            \frac{1}{{(1+|X|)}^{\rho+\ell-3}}   & \text{if}~ \rho<3, \vspace{5pt} \\
            \frac{1}{{(1+|X|)}^\ell}\ln (2+|X|) & \text{if}~ \rho =3, \vspace{5pt}\\
            \frac{1}{{(1+|X|)}^\ell}            & \text{if}~ \rho>3;
      \end{array}\right.
\end{equation}
\begin{equation}\label{eq.rho_estimates_in_B_omega}
      \int_{\mathbf{B}}\frac{\omega_\rho(y)}{|x-y|^\ell}\frac{dy}{\varepsilon^{3-\rho}}=O(1)\times
      \begin{cases}
            \omega_{\rho+\ell-3}(x)    & \text{if}~\rho<3, \vspace{5pt}\\
            \omega_{\ell}\left|\ln \mu\right| (x) & \text{if}~ \rho=3, \vspace{5pt}\\
            \omega_{\ell}(x)           & \text{if}~ \rho>3;
      \end{cases}
\end{equation}
\begin{equation}
      \label{eq.d_i_d_i_star_estimates}
      \int_{\mathbf{B}}\frac{1}{d_{i^*}^\rho (x) d_0^\ell (x)} \frac{dx}{\varepsilon^3}=O(1)\times
      \begin{cases}
            \frac{1}{d_{i^*}^{\rho+\ell-3}}(r\mathbf{m}_0) & \text{if}~ \rho<3, \vspace{5pt}\\
            \frac{1}{d_{i^*}^{\ell}}(r\mathbf{m}_0)\left|\ln \mu\right| & \text{if}~ \rho=3, \vspace{5pt}\\
            \frac{1}{d_{i^*}^{\ell}}(r\mathbf{m}_0)\mu^\frac{\rho-3}{2} & \text{if}~ \rho>3.
      \end{cases}
\end{equation}
\end{lemma}
\begin{lemma}\label{lemma.rho_and_more_estimates}
      For any $\rho\geq \frac{1}{2}$,
      and $\eta\geq 0$, we have
      \begin{align}
            \label{eq.omega_rho_omega_rho_star_estimates}
            \omega_\rho^*=O(1)\times
            \begin{cases}
                  \mu^\frac{\rho}{2}        & \text{if}~\rho>1, \\
                  \mu^\frac{1}{2} \left|\ln \mu\right| & \text{if}~\rho=1, \\
                  \mu^{\rho-\frac{1}{2}}    & \text{if}~\rho<1;
            \end{cases}
            \quad \omega_\rho=O(1)\times
            \begin{cases}
                  \omega_\ell^{\frac{\rho}{\ell}}       & \text{if}~ \ell \leq \rho, \\
                  \omega_\ell^{\frac{2\rho-1}{2\ell-1}} & \text{if}~ \ell> \rho,          \\
                  1                                           & \text{if}~ \rho=\frac{1}{2};
            \end{cases}
      \end{align} 
      and 
      \begin{align}
            \label{eq.f_norm_estimates}
            \frac{1}{2\lambda^\frac{1}{2}} \left\Vert f\right\Vert_{\frac{1}{2}}\leq
            \left\Vert f\right\Vert_{L^\infty (\mathbf{B})}\leq C \left\Vert f\right\Vert_\rho,
            \quad \left\Vert f\right\Vert_\rho\leq\left\Vert f\right\Vert_{\rho+\eta}\leq {\left(\frac{2\lambda}{\mu}\right)}^\eta\left\Vert f\right\Vert_{\rho}.
      \end{align}  
\end{lemma}

Furthermore, using Lemmas \ref{lemma.rho_and_integrals}-\ref{lemma.rho_and_more_estimates}, we get.

\begin{lemma}
      Recall that \[\Omega_i:=\Bigl\{x\in \mathbf{B}:x\cdot \mathbf{m}_i=\max_{0\leq j\leq k-1} (x \cdot \mathbf{m}_j)\Bigr\}, ~\text{for}~ i=0,1,\cdots,k-1.\]
      For $\rho\in \mathbb{R}$, we have
      \begin{align}            \label{eq.d_i_comparing_d_i_star_estimates}
            \begin{array}{cl}
                  \displaystyle \frac{C_0}{\mu^{\frac{1}{2}}}\leq d_i \leq d_{i^*} \leq C_1 d_i, \quad \Bigl|\frac{1}{d_i^\rho}-\frac{1}{d_{i^*}^\rho}\Bigr|\leq C\frac{1}{\mu^{\frac{1}{2}}d_{i^*}^{\rho+1}},
                  &\quad \text{in}~ \overline{\mathbf{B}}\backslash\overline{\Omega}_i;\vspace{8pt}\\
                  \displaystyle  d_i \leq d_{i^*}, \quad \frac{C_2}{\mu^{\frac{1}{2}}}\leq d_{i^*}, & \quad \text{in}~ \overline{\Omega}_i,
            \end{array}
      \end{align}
      where $C_0$, $C_1$, $C_2$ are some constants independent of $(\lambda,\sigma,\mu)$.
\end{lemma}
\begin{proof}
      By direct calculation, we have $d_i \leq d_{i^*}, \quad \frac{C_2}{\mu^{\frac{1}{2}}}\leq d_{i^*}$ in $\overline{\Omega}_i$ and \[\frac{C_0}{\mu^{\frac{1}{2}}}\leq d_i \leq d_{i^*} \leq C_1 d_i, ~\text{in}~ \overline{\mathbf{B}}\backslash\Omega_i.\]
      Moreover, we have $|d_{i^*}-d_i|\leq \frac{C}{\mu^{\frac{1}{2}}}$ in $\overline{\mathbf{B}}\backslash\Omega_i$. Thus $\Bigl|\frac{1}{d_i^\rho}-\frac{1}{d_{i^*}^\rho}\Bigr|\leq \frac{C}{\mu^{\frac{1}{2}}d_{i^*}^{\rho+1}}$ in $\overline{\mathbf{B}}\backslash\Omega_i$. 
\end{proof}

\begin{lemma}
      For any $\eta>0$, it holds
      \begin{align}\label{eq.inverse_laplacian_estimates}
            \bigl\Vert{(-\Delta)}^{-1}f\bigr\Vert_\rho=O(\mu^2) \times
            \begin{cases}
                  \Vert f \Vert_{\rho+2}          \vspace{5pt}                         & \text{if}~ 0<\rho<1, \\
                  \min \left\{\left\Vert f \right\Vert_3 \left|\ln \mu\right|,\left\Vert f\right\Vert_{3+\eta}\right\} & \text{if}~\rho=0.
            \end{cases},
      \end{align}
      where ${(-\Delta)}^{-1}$ is the inverse of $-\Delta$ subject to the homogeneous Dirichlet boundary condition on $\mathbf{B}$.
\end{lemma}
\begin{proof}
      Since the Green's function is bounded by $C|x-y|^{-1}$, by \eqref{eq.rho_estimates_in_B_omega}, we get the result as in \cite{Hao2015}.
\end{proof}

\section{Energy estimates}\label{appendix.energy_estimates}
\ \par

Now we show the proof of Lemma \ref{lemma.energy_estimates}. Let $\sigma_0 :=  \min \left\{ 1,\sigma\right\}$. We also introduce
\begin{align}\label{eq.B_def}
      \mathbf{B}_0 :=  \left\{x:|x-r\mathbf{m}_0|<\sigma_0 \mu^\frac{1}{2}\right\}\subset\Omega_0.
\end{align}

\begin{proof}[Proof of Lemma \ref{lemma.energy_estimates}]
\ \vspace{10pt}
\newline
\indent
      For $t=\lambda$ or $t=\sigma$, we have

      \begin{align}\label{eq.partial_J_estimates}
            \begin{aligned}
                  \partial_t J&=\int_{\mathbf{B}}\left[\nabla W \cdot \nabla W_t -K(x)W^{5+\mu}W_t\right]dx\\
                  &=-\int_{\mathbf{B}}F W_t dx=-k\int_{\Omega_0}F\bigl((U_{0})_t+R_t\bigr)dx\\
                  &=-k\int_{\Omega_0}FR_t dx-k\int_{\Omega_0 \backslash \mathbf{B}_0}F (U_{0})_t dx
                  -k\int_{\mathbf{B}_0}F (U_{0})_t dx.            
            \end{aligned}
      \end{align}

      Using \eqref{eq.F_estimates} with $\rho=2$, \eqref{eq.R_estimates_parameters}, \eqref{eq.d_i_d_i_star_estimates}, and \eqref{eq.omega_rho_omega_rho_star_estimates}, we obtain
      \begin{align*}
            k\int_{\Omega_0}FR_t dx=O\biggl(\int_{\Omega_0}\frac{\left|\ln \mu\right|}{d_0^4}\Bigl(\frac{1}{d_{0^*}}+\frac{1}{\mu^\frac{1}{2}}\sum_{i=1}^{k-1}\frac{1}{d_{i^*}^2}\Bigr)\frac{dx}{\varepsilon^3}\biggr)=O\bigl(\mu^\frac{1}{2}\left|\ln \mu\right| \bigr).
      \end{align*}
      Next, since $\varepsilon d_0>\frac{\sigma_0}{k}$ in $\Omega_0\backslash\mathbf{B}_0$, by \eqref{eq.U_i_lambda_sigma_estimates}, we see that $\varepsilon^\frac{1}{2}(U_{0})_t =O(d_0^{-1})$. Hence, using \eqref{eq.F_estimates} with $\rho=2$ we obtain
      \begin{equation*}
            k\int_{\Omega_0\backslash\mathbf{B}_0}F (U_{0})_t dx
            =\int_{\Omega_0\backslash\mathbf{B}_0}\frac{O(1)\left|\ln \mu\right|}{d_0^5}\frac{dx}{\varepsilon^3}
            =O\bigl(\mu\left|\ln \mu\right|\bigr).
      \end{equation*}
      To study the last integral in \eqref{eq.partial_J_estimates}, we use \eqref{eq.F_temp_estimates}. We consider $t=\lambda$ and $t=\sigma$ separately, since in $\mathbf{B}_0$, the estimates $(U_{0})_\lambda=O\bigl(\mu^{-\frac{1}{2}}d_0^{-1}\bigr)$ and $(U_{0})_\sigma=O\bigl(d_0^{-2}\bigr)$ differ by a factor of $\mu^\frac{1}{2}$.
      \par
      Taking $\rho=1$ in \eqref{eq.F_temp_estimates} and using $(U_{0})_\lambda=O\bigl(\mu^{-\frac{1}{2}}d_0^{-1}\bigr)$, we obtain that the following holds for $t=\lambda$.
      \begin{equation}\label{eq.partial_J_estimates_temp}
            \begin{aligned}
                  \partial_t J= & -k\int_{\mathbf{B}_0}\left(K(x)U_0^\mu-K(r)\right)U_0^5 (U_{0})_tdx
                  \\ &+(5+\mu)k \int_{\mathbf{B}_0}K(x)RU_0^{4+\mu}(U_{0})_t dx+O\left(\mu^\frac{1}{2}\left|\ln \mu\right|\right).
            \end{aligned}
      \end{equation}
      Note that by radial symmetry and \eqref{eq.U_i_lambda_sigma_estimates},
      \begin{equation*}
            R^2(r\mathbf{m}_0)\int_{\mathbf{B}_0}U_0^{3+\mu}(U_{0})_\sigma dx=
            \frac{\mu^\frac{1}{2}}{4}\frac{K'(r)}{K(r)}R^2(r\mathbf{m}_0)\int_{\mathbf{B}_0}U_0^{4+\mu}dx
            =O\bigl(\mu^\frac{3}{2}\bigr).
      \end{equation*}
      Hence, taking $\rho=0$ in \eqref{eq.F_temp_estimates} and using $(U_{0})_\sigma=O\bigl(d_0^{-2}\bigr)$,
      we get that \eqref{eq.partial_J_estimates_temp} holds for $t=\sigma$ as well.
      \par
      It remains to evaluate the two integrals in \eqref{eq.partial_J_estimates_temp}. From \eqref{eq.U_i_lambda_sigma_estimates} and the fact
      \begin{align*}
            \nabla U_0=\frac{1-d_0^2}{d_0^2}U_0,
      \end{align*}
we can derive the following:
      \begin{align*}
            \begin{aligned}
                   U_0^5 \partial_\lambda U_{0} & =\frac{1}{6\lambda}\divop\left((x-r\mathbf{m}_0)U_0^6\right),& U_0^4 \partial \lambda U_{0} &=\frac{1}{5\lambda}\divop\left((x-r\mathbf{m}_0)U_0^5\right)-\frac{1}{10\lambda}U_0^5,\\[5pt]
                   U_0^5 \partial_\sigma U_{0} & =\frac{\mu^\frac{1}{2}}{6}\mathbf{m}_0\cdot \nabla U_0^6
                  +\frac{\mu^\frac{1}{2}}{4}\frac{K'(r)}{K(r)}U_0^6, &
                   U_0^4 \partial \sigma U_{0} &=\frac{\mu^\frac{1}{2}}{5}\mathbf{m}_0\cdot \nabla U_0^5
                  +\frac{\mu^\frac{1}{2}}{4}\frac{K'(r)}{K(r)}U_0^5.
            \end{aligned}
      \end{align*}
      Note that $a :=  \left(U_0^\mu-1\right)=O\bigl(\mu^\frac{1}{2} \left|\ln \mu\right| \bigr)$ in $\mathbf{B}_0$, and
      $k := \bigl\lfloor \mu^{-\frac{1}{2}}\bigr\rfloor=O\bigl(\mu^{-\frac{1}{2}}\bigr)$. 
      First, we calculate
      \begin{align*}
            \begin{aligned}
            & k\int_{\mathbf{B}_0}  \left(K(x)U_0^\mu-K(r)\right)U_0^{5+\mu} (U_{0})_\lambda dx  \\
            &\quad =-(a+1)\frac{k}{6\lambda}\int_{\mathbf{B}_0}\left[\left(K(x)-K\left(r\right)\right)+aK(x)\right]\divop\left(\left(x-r\mathbf{m}_0\right)U_0^6\right)dx \\
            &\quad =O(1)\times\biggl[\mu^\frac{1}{2}\left|\ln \mu\right|\int_{\partial \mathbf{B}_0}U_0^{6}dS_x
                  -k\int_{\mathbf{B}_0}\left(\frac{K'(x)}{|x|}-K'(1)+K'(1)\right)x\cdot (x-r\mathbf{m}_0)U_0^6dx \biggr] \\
            &\quad =O(\mu)+O(1)k\int_{\mathbf{B}_0}\left(r\mathbf{m}_0\cdot (x-r\mathbf{m}_0)+O(\mu) \right)U_0^6dx=O\bigl(\mu^\frac{1}{2}\bigr).
            \end{aligned}
      \end{align*}
      Here we use the fact $U_0\leq C$ on $\partial \mathbf{B}_0$ and that $U_0$ is a function of $|x-r\mathbf{m}_0|$. Similarly, with the fact $|x-r\mathbf{m}_0|=\sigma_0 \mu^\frac{1}{2}$, $R=O\bigl(\mu^\frac{1}{2}\bigr)$, $U_0 \leq C$ on $\partial \mathbf{B}_0$ and $|\nabla R|=O\bigl(\mu^{-\frac{1}{2}}\bigr)$, $|x-r\mathbf{m}_0|\leq \varepsilon d_0$, in $\mathbf{B}_0$,
       we have
      \begin{align*}
            \begin{aligned}
                  &k\int_{\mathbf{B}_0} \left(K(x)U_0^\mu-K(r)\right)U_0^{5+\mu} (U_{0})_\sigma dx  \\
                 &\quad =O(\mu^\frac{1}{2}\left|\ln \mu\right|)-(1+a)\frac{1}{6}\int_{\mathbf{B}_0}(K(x)U_0^\mu-K(r))\mathbf{m}_0\cdot \nabla U_0^6dx  \\
                 &\quad =\frac{K'(1)}{6}\int_{\mathbb{R}^3}\Phi^6(y)dy+O \bigl(\mu^\frac{1}{2}\left|\ln \mu\right| \bigr)\\
                 &\quad =K'(1)A_1+O\bigl(\mu^\frac{1}{2}\left|\ln \mu\right|\bigr),
            \end{aligned}
      \end{align*}
      and
      \begin{align*}
            \begin{aligned}
                  &-(5+\mu)k \int_{\mathbf{B}_0}K(x)RU_0^{4+\mu}(U_{0})_\lambda dx                         \\
                  &\quad=\frac{k}{\lambda}\left(1+O\left(\mu^\frac{1}{2}\left|\ln \mu\right|\right)\right)\int_{\mathbf{B}_0}
                  R\left(\frac{1}{2}U_0^5-\divop\left((x-r\mathbf{m}_0)U_0^5\right) \right)dx  \\
                  &\quad=\frac{k}{2\lambda}\left(1+O\left(\mu^\frac{1}{2}\left|\ln \mu\right|\right)\right)\left[\int_{\mathbf{B}_0}\left(R(r\mathbf{m}_0)+O(1)\Vert\nabla R \Vert_{L^\infty}\varepsilon d_0 \right)U_0^5dx+O(\mu)\right] \\
                  &\quad=\frac{E(r\mathbf{m}_0)}{2\lambda^2}\int_{\mathbb{R}^3}\Phi(0)\Phi^5(y)dy+O\left(\mu \left|\ln \mu\right|\right)\\
                  &\quad=\frac{2A_2}{\lambda^2}E(r\mathbf{m}_0)+O\left(\mu^\frac{1}{2}\left|\ln \mu\right|\right).
            \end{aligned}
      \end{align*}
      Finally, using the fact $\mathbf{m}_0\cdot \nabla R = \mathbf{m}_0\cdot \nabla R(r\mathbf{m}_0)+O(d_0)$, we get
      \begin{align*}
            \begin{aligned}
                  & -(5  +\mu)k \int_{\mathbf{B}_0}K(x)RU_0^{4+\mu} (U_{0})_\sigma dx \\
                      & \quad =-\left[1+O\left(\mu^\frac{1}{2}\left|\ln \mu\right|\right)\right]\int_{\mathbf{B}_0}R\left(\frac{5}{4}\frac{K'(r)}{K(r)}U_0^5+\mathbf{m}_0\cdot\nabla U_0^5 \right)dx                          \\
                      & \quad =\left[1+O\left(\mu^\frac{1}{2}\left|\ln \mu\right|\right)\right]\int_{\mathbf{B}_0}\left(\nabla R\cdot \mathbf{m}_0-\frac{5}{4}\frac{K'(r)}{K(r)}R  \right)U_0^5dx+O\left(\mu^\frac{1}{2}\left|\ln \mu\right|\right) \\
                      & \quad =\mu^\frac{1}{2}\frac{\mathbf{m}_0\cdot \nabla E(r\mathbf{m}_0)}{\lambda^3}\int_{\mathbb{R}^3}\Phi(0)\Phi^5(y)dy+O\left(\mu^\frac{1}{2}\left|\ln \mu\right|\right).
            \end{aligned}
      \end{align*}
      In conclusion, we have
      \begin{align*}
            \begin{aligned}
                  \partial_\lambda J &=\frac{2A_1}{\lambda^2}E(r\mathbf{m}_0)+O\left(\mu^\frac{1}{2}\left|\ln \mu\right|\right),\\
                  \partial_\sigma J &=K'(1)A_1+\frac{4A_2\mu^\frac{1}{2}}{\lambda}\mathbf{m}_0\cdot \nabla E(r\mathbf{m}_0)+O\left(\mu^\frac{1}{2}\left|\ln \mu\right|\right).
            \end{aligned}
      \end{align*}
      Collecting above all and \eqref{eq.E_x_0_and_nabla_E_x_0}, we obtain the estimates of the first-order partial derivative.
      By the similar method, we obtain the remaining two estimates. Thus, we complete the assertion of the lemma.
\end{proof}

\end{document}